\pdfoutput=1
\documentclass[11pt,a4paper,reqno]{amsart}
\usepackage[includehead,includefoot,margin=25mm]{geometry}
\usepackage{times}
\usepackage{amsfonts} %
\usepackage{amsmath} %
\usepackage{amssymb} %
\usepackage{amsthm} %
\usepackage{enumerate} %
\usepackage{bbm} %
\usepackage{graphicx} %
\usepackage{nicefrac} %
\usepackage{color} %
\usepackage{hyperref}
\usepackage{tikz-cd}
\usepackage{mathrsfs}

\allowdisplaybreaks

\usepackage{graphicx} 
\usepackage{tikz}

\usetikzlibrary{circuits.logic.US,circuits.logic.IEC,fit}

\hypersetup{
	colorlinks=true,       
	linkcolor=blue,          
	citecolor=blue,        
	filecolor=blue,      
	urlcolor=blue           
}
\usepackage[all]{xy}
\usepackage{calrsfs} 

\usepackage{imakeidx}
\makeindex

\theoremstyle{plain}
\newtheorem{thm}{Theorem}[section]
\theoremstyle{plain}
\newtheorem{conj}[thm]{Conjecture}
\newtheorem{lem}[thm]{Lemma}
\theoremstyle{definition}
\newtheorem{defi}[thm]{Definition}
\theoremstyle{definition}
\newtheorem{example}[thm]{Example}
\theoremstyle{plain}
\newtheorem{prop}[thm]{Proposition}
\theoremstyle{plain}
\newtheorem{cor}[thm]{Corollary}
\theoremstyle{remark}
\newtheorem{notation}[thm]{Notation}
\theoremstyle{remark}
\newtheorem{remark}[thm]{Remark}

\newcommand{\Z}{\operatorname{Z}}
\newcommand{\CE}{\mathcal{E}}
\usepackage{scalerel,stackengine}
\stackMath
\newcommand\bbhat[1]{%
	\savestack{\tmpbox}{\stretchto{%
			\scaleto{%
				\scalerel*[\widthof{\ensuremath{#1}}]{\kern-.6pt\bigwedge\kern-.6pt}%
				{\rule[-\textheight/2]{1ex}{\textheight}}
			}{\textheight}%
		}{0.5ex}}%
	\stackon[1pt]{#1}{\tmpbox}%
}
\newcommand{\PP}[0]{\ensuremath{\mathbb{P}}}
\newcommand{\CC}[0]{\ensuremath{\mathbb{C}}}
\newcommand{\ZZ}[0]{\ensuremath{\mathbb{Z}}}

\newcommand{\QQ}[0]{\ensuremath{\mathbb{Q}}}
\newcommand{\RR}[0]{\ensuremath{\mathbb{R}}}
\newcommand{\GG}[0]{\ensuremath{\mathbb{G}}}
\newcommand{\pp}{\mathfrak{p}}

\newcommand{\GL}[0]{\ensuremath{\operatorname{GL}}}

\newcommand{\CO}[0]{\ensuremath{\mathcal{O}}}

\newcommand{\Pic}[0]{\ensuremath{\operatorname{Pic}}}

\newcommand{\Cone}{\ensuremath{\operatorname{cone}}}

\newcommand{\Spec}{\operatorname{Spec}}

\newcommand{\ord}{\operatorname{ord}}

\newcommand{\rk}{\operatorname{rk}}

\newcommand{\vol}{\operatorname{vol}}

\newcommand{\Div}{\operatorname{Div}}

\newcommand{\Hom}{\operatorname{Hom}}
\newcommand{\Val}{\operatorname{Val}}
\newcommand{\TT}{\mathbb{T}}
\newcommand{\OL}[1]{\overline{#1}}
\newcommand{\N}{\operatorname{N}}

\newcommand{\hdeg}{\widehat{\operatorname{deg}}}

\usepackage[normalem]{ulem}
\usepackage{dynkin-diagrams}
\usepackage{multirow}

\makeatletter
\let\@wraptoccontribs\wraptoccontribs
\makeatother

\begin{document}
	
	\title[]{Counting rational points on Hirzebruch--Kleinschmidt varieties over number fields}
	
	\author{Sebastián Herrero}
\address{Universidad de Santiago de Chile, Dept.~de Matem\'atica y Ciencia de la Computaci\'on, Av.~Libertador Bernardo O'Higgins 3363, Santiago, Chile, and ETH, Mathematics Dept., CH-8092, Z\"urich, Switzerland}
  \email{sebastian.herrero.m@gmail.com}
  
  \author{Tobías Martínez}
\address{Departamento de Matem\'aticas, Universidad T\'ecnica
  Fe\-de\-ri\-co San\-ta Ma\-r\'\i a, Av.~Espa\~na 1680, Valpara\'\i
  so, Chile}
\email{tobias.martinez@usm.cl}

\author{Pedro Montero}
\address{Departamento de Matem\'aticas, Universidad T\'ecnica
  Fe\-de\-ri\-co San\-ta Ma\-r\'\i a, Av.~Espa\~na 1680, Valpara\'\i so, Chile}  \email{pedro.montero@usm.cl}

	\begin{abstract}
		We study the asymptotic growth of the number of rational points of bounded height on smooth projective split toric varieties with Picard rank 2 over number fields, with respect to Arakelov height functions associated with big metrized line bundles. We show that these varieties can be naturally decomposed into a finite disjoint union of subvarieties, where explicit asymptotic formulas for the number of rational points of bounded height can be given. Additionally, we present various examples, including the case of Hirzebruch surfaces.
	\end{abstract}

 
	\maketitle


 
	
	\setcounter{tocdepth}{1}
	\tableofcontents
	
	\section{Introduction}

In~\cite{Sch79}, Schanuel considered the projective space~$\PP^n$ over a number field~$K$ with height function~$H([x_0,\ldots,x_n]):=\prod_{v\in \Val(K)} \max \{|x_i|_v\}$ and proved that the number 
$$N(\PP^n, H, B):=\#\{P\in \PP^n(K):H(P)\leq B\},$$
of rational points of height bounded by~$B>0$, has the following asymptotic behavior:
$$N(\PP^n, H,B)=CB^{n+1}+\left\lbrace \begin{array}{ll}
O(B \log B)	& \text{if } n=1, \\
O(B^n)	& \text{if } n>1,
\end{array}\right. \quad \text{as } B\rightarrow \infty, $$
where $C=C(K,n)$ is an explicit constant depending on the field~$K$ and the dimension~$n$. In this case, the height function~$H$ coincides with the one obtained by metrizing the line bundle $\CO_{\PP^n}(1)$ in the sense of \cite{Fran/Man/Tsch89}.
With this observation and certain calculations performed on a cubic surface by Manin in \cite[Appendix]{Fran/Man/Tsch89}, it was conjectured that on a Fano variety $X$ with a dense set of rational points $X(K)$, excluding some degenerate cases, when considering the height function $H$ induced by a natural metrization of the anticanonical line bundle~$-K_X$, the number~$N(X,H,B)$ of rational points of height bounded by~$B$ should satisfy
$$N(X,H,B)\sim C B(\log B)^{\operatorname{rk}\Pic(X) - 1} \quad \text{as }B\rightarrow \infty,$$ 
for some constant~$C>0$.

In 1996, Batyrev and Tschinkel \cite{BatyrevTschinkelRationalPoints} provided a counterexample to the first version of Manin's conjecture. Nowadays, the current expectation is as follows (see e.g. \cite[{\it Formule empirique} 5.1]{Peyre2003PHBTA} or \cite[Conjecture~6.3.1.5]{Arzhantsev_Derenthal_Hausen_Laface_2014CoxRings} for a detailed discussion and the precise definition of the relevant concepts).

\begin{conj}[Manin--Peyre]
Let $X$ be an almost Fano variety\footnote{Following~\cite[{\it D\'efinition}~3.1]{Peyre2003PHBTA}, an almost Fano variety is a smooth, projective, geometrically integral variety~$X$ defined over a field~$K$, with~$H^1(X, \CO_X) = H^2(X, \CO_X) = 0$, torsion-free geometric Picard group~$\Pic(X_{\overline{K}})$ and $-K_X$ big.} over
a number field $K$, with dense set of rational points $X(K)$, finitely generated~
$\Lambda_{\textup{eff}}(X_{\overline{K}})$ and trivial Brauer group $\operatorname{Br}(X_{\overline{K}})$. Let $H=H_{-K_X}$ be the anticanonical height
function, and assume that there is an open subset $U$ of $X$ that is the complement
of the weakly accumulating subvarieties on $X$ with respect to $H$. Then, there is a constant $C>0$  such that
$$N(U, H_{-K_X},B) \sim C B(\log B)^{\rk \operatorname{Pic}(X)-1} \quad \text{as } B\rightarrow \infty.$$
Moreover, the leading constant is of the form
$$C=\alpha(X)\beta(X)\tau_H(X),$$
where
\begin{align*}
       \alpha(X)&:=\frac{1}{\left(\rk \Pic(X)-1\right)!}\int_{\Lambda_{\textup{eff}}(X)^\vee} e^{-\left \langle -K_X,\mathbf{y}\right \rangle }\textup{d}\mathbf{y},  \\
       \beta(X)&:=\#H^1(\textup{Gal}(\overline{K}/K), \Pic(X_{\overline{K}})), 
\end{align*}
and~$\tau_H(X)$ is the Tamagawa number of $X$ with respect to $H$ (as defined by Peyre in~\cite[Section~4]{Peyre2003PHBTA}).
\end{conj}

There is also a conjecture, originating in the work of Batyrev and Manin~\cite{Bat/Man90}, concerning the asymptotic growth of the number
$$N(U,H_L,B)=\#\{P\in U(K):H_L(P)\leq B\},$$
when considering height functions~$H_L$ associated to big metrized line bundles~$L$ on a variety~$X$ as above, and for appropriate open subsets~$U\subseteq X$. More precisely, if we denote by $\tau \prec \sigma$ whenever $\tau$ is a face of a cone $\sigma$, then one defines  the following classical numerical invariants for a line bundle class~$L$ on $X$:
\begin{equation*}
    \begin{split}
       a(L)&:=\inf\{a\in \RR: aL+K_X\in \Lambda_{\textup{eff}}(X)\}, \\
       b(L)&:=\max\{\textup{codim}(\tau): a(L)L+K_X\in \tau\prec \Lambda_{\textup{eff}}(X)\},
    \end{split}
\end{equation*}
which measure the position of $L$ inside the effective cone~$\Lambda_{\textup{eff}}(X)$. With this notation, the more general version of the above conjecture states that there exists a constant~$C>0$ such that
\begin{equation}\label{eq Batirev-Manin conjecture}
  N(U,H_L,B)\sim CB^{a(L)}(\log B)^{b(L)-1} \quad \text{as }B\to \infty.  
\end{equation}
We refer the reader to~\cite{Bat/Tsch98} for extensions, and a conjectural description of the leading constant~$C$ in terms of geometric, cohomological and adelic invariants associated to $U$ and~$L$.

The above conjectures have been proven by various authors, either in specific examples or in certain families of varieties (see for instance~\cite{Peyre2003PHBTA}, \cite[Section~3]{Tschinkel03}, \cite{Bro07} and~\cite[Section~4]{Tsch008} for accounts of such results). 

The motivation for the present paper stems from the groundbreaking work of Batyrev and Tschinkel \cite{Bat/Tsch98}, who demonstrated the Manin--Peyre conjecture for smooth projective toric varieties using harmonic analysis techniques, with~$U$ the dense toric orbit and~$H$ the anticanonical height function. We also refer to~\cite{Bat/Tsch96} for a similar result concerning~\eqref{eq Batirev-Manin conjecture} (for some~$C>0$) on toric varieties with height functions associated to  big metrized line bundles.

In this paper, we focus on smooth projective split toric varieties with Picard rank 2 over a number field~$K$, and consider Arakelov height functions associated to  big metrized line bundles. Our results show that these varieties can be naturally decomposed into a finite disjoint union of subvarieties where explicit asympotic formulas for the number of rational points of bounded height can be given, in the spirit of Schanuel's work~\cite{Sch79}. Hence, in this setting, we go beyond the scope of the classical conjectures mentioned above. We achieve these results by using suitable algebraic models for such varieties and by performing explicit computations on the associated height zeta functions. 

In order to be more precise, we recall a geometric result due to Kleinschmidt~\cite{Kleinschmidt88} stating that all smooth projective toric varieties of Picard rank 2 are (up to isomorphism) of the form 
\[
X=\PP(\CO_{\PP^{t-1}}\oplus \CO_{\PP ^{t-1}}(a_1)\oplus\cdots\oplus\CO_{\PP ^{t-1}}(a_r)),
\]
where $r \geq 1$, $t \geq 2$, and $0 \leq a_1 \leq \cdots \leq a_r$ are integers. For computational convenience, we choose a different normalization and put
$$X_d(a_1, \ldots, a_r):=\PP(\CO_{\PP^{t-1}}\oplus \CO_{\PP ^{t-1}}(-a_r)\oplus\cdots\oplus\CO_{\PP ^{t-1}}(a_{r-1}-a_r))\simeq X,$$
with $d:= r + t - 1$ the dimension of~$X_d(a_1, \ldots, a_r)$. We refer to these varieties as \emph{Hirzebruch--Kleinschmidt varieties}.

If~$X=X_d(a_1, \ldots, a_r)$, then we consider the projective subbundle~$F:=\PP(\CO_{\PP ^{t-1}}(-a_r)\oplus\cdots\oplus\CO_{\PP ^{t-1}}(a_{r-1}-a_r))$ and note that~$F\simeq X_{d-1}(a_1, \ldots, a_{r-1})$ when~$r>1$, while $F\simeq \PP^{t-1}$ when~$r=1$. Then, we define the \emph{good open subset} of~$X$ as
$$U_d(a_1, \ldots, a_r):=X\setminus F.$$
We note that this open subset is larger than the dense toric orbit of~$X$.

Our main result gives an asymptotic formula for~$N(U_d(a_1, \ldots, a_r),H_L,B)$ of the form~\eqref{eq Batirev-Manin conjecture}, with an explicit constant~$C=C_{L,K}$, for every big line bundle class~$L\in \Pic(X)$ that we equip with a ``standard metrization''. This general result is given in Section~\ref{subsection:general case} (see Theorems~\ref{thm:generalVersionTheorem} and~\ref{thm:degenerate_case_general_theorem}). For simplicity, we present here the result for~$L=-K_X$, noting that the anticanonical line bundle on smooth projective toric varities is always big.

	\begin{thm}\label{main thm anticanonical}
		Let $X=X_d(a_1,\ldots,a_r)$ be a Hirzebruch--Kleinschmidt variety  of dimension $d=r+t-1$ over a number field $K$, let~$H=H_{-K_X}$ denote the anticanonical height function on~$X(K)$, and let~$U=U_d(a_1,\ldots,a_r)$ be the good open subset of~$X$. Then, we have
        $$N(U,H,B) \sim C B\log(B) \quad \text{as }B\to \infty,$$
     with
     \begin{equation}\label{eq constant main thm anticanonical}
       C:=\frac{R_K^2h_K^2|\Delta_K|^{-\frac{(d+2)}{2}}}{w_K^2 (r+1) (t+(r+1)a_r-|\mathbf{a}|) \xi_K(r+1)\xi_K(t)},  
     \end{equation}
          where~$R_K,h_K,\Delta_K$ are the regulator, class number and discriminant of~$K$, respectively, $|\mathbf{a}|:=\sum_{i=1}^ra_i$ and
     $$\xi_K(s):=\left(\frac{\Gamma(s/2)}{2\pi^{s/2}}\right)^{r_1}\left(\frac{\Gamma(s)}{(2\pi)^{s}}\right)^{r_2}\zeta_K(s),$$
     with~$r_1$ and $r_2$ the number of real and complex Archimedean places of~$K$, respectively, and~$\zeta_K$ the Dedekind zeta function of~$K$.
     \end{thm}

\begin{remark}
In Section~\ref{subsection Peyre's alpha} we show that
$$\alpha(X)=\frac{1}{(r+1) (t+(r+1)a_r-|\mathbf{a}|)}.$$
Also, since $X$ is a split toric variety, we have $\beta(X)=1$ (this is well-known and follows e.g.~from~\cite[Remark~1.7 and Corollary~1.18]{Bat/Tsch98} and~\cite[Lemma~2.21]{PRR23}). Hence, Theorem~\ref{main thm anticanonical} together with the main theorem of~\cite{Bat/Tsch98} implies that the Tamagawa number of~$X$ with respect to the anticanonical height function~$H_{-K_X}$ is
$$\tau_H(X)=\frac{R_K^2h_K^2|\Delta_K|^{-\frac{(d+2)}{2}}}{w_K^2  \xi_K(r+1)\xi_K(t)}.$$
\end{remark}
    
     Since the good open subset~$U$ in Theorem~\ref{main thm anticanonical} is obtained by removing from~$X_d(a_1,\ldots,a_r)$ the subbundle~$F$, which is either another Hirzebruch--Kleinschmidt variety or a projective space, we can also describe the asymptotic behaviour of the number~$N(F,H,B)$, provided the restriction~$-K_X|_F$ is big in~$\Pic(F)$ (it is easily seen that when this is not the case, one has~$N(F,H,B)=\infty$ for every~$B\geq 1$). In general, the restriction~$-K_X|_F$ does not coincide with the anticanonical class~$-K_F$ of~$F$ (see Section~\ref{sec restriction of divisors}), and this is the main reason why we are lead to study the asymptotic behaviour of~$N(U,H_L,B)$ for height functions~$H_L$ associated to general big line bundle classes. These ideas are illustrated in the following example.

     \begin{example}\label{example HK threefold}
     Given integers~$0\leq a_1\leq a_2$ consider the Hirzebruch--Kleinschmidt threefold
         $$X:=X_3(a_1,a_2)=\PP(\CO_{\PP^{1}}\oplus \CO_{\PP ^{1}}(-a_2)\oplus\CO_{\PP ^{1}}(a_1-a_2)),$$
         with projection map~$\pi:X\to \PP^{1}$ and good open subset~$U=U_3(a_1,a_2)$. 
         Using results from Sections~\ref{subsection effective divisors} and~\ref{sec restriction of divisors}, we get that~$-K_X=\CO_X(3)\otimes \pi^*(\CO_{\PP^1}(2+2a_2-a_1))$, and the restriction~$-K_X|_F$ of the anticanonical divisor of~$X$ to the projective subbundle~$F=\PP( \CO_{\PP ^{1}}(-a_2)\oplus\CO_{\PP ^{1}}(a_1-a_2))$ corresponds, under the isomorphism~$F\simeq X_2(a_1)=:X'$, to the line bundle~$L:=\CO_{X'}(3)\otimes (\pi')^*(\CO_{\PP^1}(2+2a_1-a_2))$, where~$\pi':X'\to \PP^1$ is the corresponding projection map. If we write~$X'=U'\sqcup F'$ with~$U'=U_2(a_1)$ the good open subset of~$X'$ and~$F'$ its complement in~$X'$, then we have~$F'\simeq \PP^1$, and the restriction~$L|_{F'}$ corresponds to the line bundle~$M:=\CO_{\PP^1}(2-a_1-a_2)$. Hence, we have a disjoint decomposition
         \begin{equation}\label{eq decomposition of X first example}
            X\simeq U \sqcup U'\sqcup F' \simeq U_3(a_1,a_2) \sqcup U_2(a_1) \sqcup \PP^1, 
         \end{equation}
         and denoting by~$H=H_{-K_X}$ the anticanonical height function, we obtain 
         $$N(U,H,B)\sim CB\log(B) \quad \text{as }B\to \infty,$$
         with an explicit constant~$C>0$ by Theorem~\ref{main thm anticanonical}. Now, by Lemma~\ref{lem:restrictionBigDivisor} in Section~\ref{sec restriction of divisors} the line bundle~$L$ is big if and only if~$a_2<2a_1+2$, in which case Theorem~\ref{thm:generalVersionTheorem} in Section~\ref{section:rational points} gives
         $$N(U',H,B)=N(U_2(a_1),H_L,B)\sim C' B^{\frac{a_1+2}{2a_1+2-a_2}} \quad \text{as }B\to \infty,$$
         with another explicit constant~$C'>0$. Finally, by Lemma~\ref{lem:restrictionBigDivisorII} $M$ is big if and only if~$a_1+a_2<2$, in which case Schanuel's estimate, in the form of Corollary~\ref{cor:SchanuelTheorem} in Section~\ref{sec: height zeta function projective}, gives
         $$N(F',H,B)=N(\PP^1,H_M,B)\sim C''B^{\frac{2}{2-a_1-a_2}} \quad \text{as }B\to \infty,$$
         with yet another explicit constant~$C''>0$. We can then distinguish several different cases in order to compare the contribution of each of the subsets~$U,U',F'$ in the decomposition~\eqref{eq decomposition of X first example} to the asymptotic growth of the number of rational points of bounded anticanonical height on~$X$, as represented in the following table.


 \begin{center}
         \begin{tabular}{|c|c|c|c|}
         \hline 
         \textbf{Case} & \textbf{Is~$L$ big?} & \textbf{Is~$M$ big?}  & \textbf{Comparison}\\
         \hline 
         $(a_1,a_2)=(0,0)$ &  Yes & Yes & \begin{minipage}{5.5cm}
                                     $N(F',H,B)=o(N(U',H,B))$  \\
                                     $N(U',H,B)=o(N(U,H,B))$ \\
                                     $N(U,H,B)\sim C B\log(B)$  
                                     \end{minipage} \\
                                  \hline
         $(a_1,a_2)=(0,1)$ &  Yes & Yes &  
                \begin{minipage}{5.5cm}
                $N(U,H,B)=o(N(U',H,B))$\\
                $N(U',H,B)\sim C' B^2$\\
                $N(F',H,B)\sim C'' B^2$
                 \end{minipage}
        \\
          \hline
         $1\leq a_1<a_2<2a_1+2$ & Yes & No & 
         \begin{minipage}{5.5cm} 
         $N(F',H,B)=\infty$\\
         $N(U,H,B)=o(N(U',H,B))$\\
         $N(U',H,B)\sim C' B^{\frac{a_1+2}{2a_1+2-a_2}}$
         \end{minipage} \\
         \hline
         $1\leq a_1=a_2 $ &  Yes & No & 
         \begin{minipage}{5.5cm}
         $N(F',H,B)=\infty$\\
         $N(U',H,B)=o(N(U,H,B))$ \\
         $N(U,H,B)\sim CB\log(B)$
         \end{minipage}\\
         \hline
          $2a_1+2\leq a_2$ &  No & No & \begin{minipage}{5.5cm}$N(F',H,B)=\infty$\\
          $N(U',H,B)=\infty$\\
          $N(U,H,B)\sim CB\log(B)$
          \end{minipage}\\
         \hline
                 \end{tabular}
                 \end{center}

     
         Since the values~$r=t=2$ are fixed, the constants~$C,C',C''$ above depend only on the base field~$K$ and on the coefficients~$a_1,a_2$. In order to give a concrete numerical example, let us for simplicity assume~$K=\QQ$ and choose~$(a_1,a_2)=(0,1)$. Using that~$\xi_{\mathbb{Q}}(s)=(2\pi^{s/2})^{-1}\Gamma(s/2)\zeta(s)$, we get
         $$C=\frac{\pi^{2}}{6\zeta(3)\zeta(2)}=\frac{1}{\zeta(3)}=0.83190737\ldots,$$
         while the values of the constants~$C'$ and~$C''$ can be extracted from Example~\ref{example product of projective spaces} and Corollary~\ref{cor:SchanuelTheorem}, and are given by
         $$C'=\frac{3}{\pi}\left(1+\frac{945\zeta(3)}{16\pi^3}\right) = 3.14147564\ldots, \quad C''=\frac{\pi}{2\zeta(2)}=\frac{3}{\pi}=0.95492965\ldots.$$
         In particular, this shows that~$U'$ ``contributes more'' than~$F'$ to the number of rational points of bounded anticanonical height on~$X_3(0,1)$.
           \end{example}

In the general case, our strategy to study the number of rational points of bounded height on a Hirzebruch--Kleinschmidt varierty is as follows: Given~$X=X_d(a_1,\ldots,a_r)$ 
we start by writing
\begin{equation}\label{eq natural decomposition HK}
  X\simeq U_d(a_1,\ldots,a_r)\sqcup U_{d-1}(a_1,\ldots,a_{r-1})\sqcup \cdots \sqcup U_{t}(a_1) \sqcup \PP^{t-1}.  
\end{equation}
Then, starting with the anticanonical height function~$H_{-K_X}$ on~$X(K)$, we give simple criteria to decide if the induced height functions~$H_i$ on~$U_{t+i-1}(a_1,\ldots,a_{i})(K)$ (for~$1\leq i\leq r$), and~$H_{0}$ on~$\PP^{t-1}(K)$, are associated to big line bundle classes in the corresponding Picard groups. Finally, using Theorems~\ref{thm:generalVersionTheorem} and~\ref{thm:degenerate_case_general_theorem} in Section~\ref{subsection:general case}, together with Schanuel's estimate (Corollary~\ref{cor:SchanuelTheorem} in Section~\ref{sec: height zeta function projective}), we can give explicit asymptotic formulas for the numbers~$N(U_{d+i-r}(a_1,\ldots,a_{i}),H_i,B)$ and~$N(\PP^{t-1},H_0,B)$, obtaining the counting of rational points of bounded height on each piece of the decomposition~\eqref{eq natural decomposition HK}. This approach also works if we start with a height function~$H_L$ on~$X(K)$ associated to any~$L\in \Pic(X)$ big.

\begin{remark}
 Assume~$a_1=\ldots=a_j=0$ and~$0<a_{j+1}\leq \ldots \leq a_r$, for some~$j\in \{1,\ldots,r-1\}$. In this case, instead of~\eqref{eq natural decomposition HK}, one can rather write
 \begin{equation}\label{eq natural decomposition HKII}
    X\simeq U_d(a_1,\ldots,a_r)\sqcup U_{d-1}(a_1,\ldots,a_{r-1})\sqcup \cdots \sqcup U_{t+j}(a_1,\ldots,a_{j+1}) \sqcup (\PP^{t-1}\times \PP^j), 
 \end{equation}
 and proceed as above by giving explicit asymptotic formulas for the numbers~$N(U_{d+i-r}(a_1,\ldots,a_{i}),H_i,B)$ (for~$j+1\leq i\leq r$) and~$N(\PP^{t-1}\times \PP^j,H_j,B)$. This variant seems more natural to us in this case since there is no gain in removing a closed subvariety of~$X_{t+j-1}(a_1,\ldots,a_j)=\PP^{t-1}\times \PP^j$ when counting rational points of bounded height on this particular component of the decomposition~\eqref{eq natural decomposition HKII}. For instance, in the case of~$X_3(0,1)$ discussed in Example~\ref{example HK threefold}, one can write~$X_3(0,1)\simeq U_3(0,1)\sqcup (\PP^1\times \PP^1)$ and compute~$N(\PP^1\times \PP^1,H_L,B)$ directly using Theorem~\ref{thm:degenerate_case_general_theorem} in Section~\ref{subsection:general case}.
\end{remark}

We refer the reader to Section~\ref{subsection hirzebruch surfaces} where we apply the above strategy to the case of Hirzebruch surfaces~$X=X_2(a)$ with~$a>0$ an integer. In particular, we recover a classical example going back to Serre~\cite{Serre1989}, and revisited by Batyrev and Manin in~\cite{Bat/Man90} and by Peyre in~\cite{Pey02} (see Remark~\ref{remark Serre and BM}).

Our proof of Theorem~\ref{main thm anticanonical} (and its extension to general big line bundles) is based on the analytic properties of the associated height zeta functions, which we relate to height zeta functions of projective spaces and to the zeta function~$\xi_K$ of the base field~$K$. Then, as usual, a direct application of a Tauberian theorem leads to the desired results. The analytic continuation and identification of the first pole of our height zeta functions are achieved through explicit computations and by exploiting the concrete algebraic models of our Hirzebruch--Kleinschmidt varieties. As such, it would be interesting to investigate whether the techniques used in this paper can be applied to other families of algebraic varieties.

\begin{remark}
\begin{enumerate}
\item Most of the geometric ideas presented in this work, particularly the 
decompositions~\eqref{eq natural decomposition HK} and~\eqref{eq natural decomposition HKII}, stem from our research on the analogous problem in the case of Hirzebruch--Kleinschmidt varieties over global function fields, which is studied in detail in the companion paper~\cite{HMM24b}.
\item When working on this project, we came across an unpublished manuscript by Maruyama~\cite{Maruyama2015}, which proposes an asymptotic formula for the number of rational points of bounded anticanonical height on Hirzebruch surfaces. Unfortunately, the proposed result is incorrect due to convergence issues with the relevant height zeta function that occur when one does not remove the corresponding subbundle~$F$ of the Hirzebruch surface~$X_2(a)$, as we have done here in greater generality. Nevertheless, we remark that the approach used in loc.~cit.~to study the analytic properties of height zeta functions of projective bundles has served as an inspiration for the present work (see Section~\ref{sec: height zeta function projective}).
\end{enumerate}
\end{remark}

This paper is organized as follows. In Section~\ref{sec preliminaries} we establish notation, which in most cases is classical in number theory and algebraic geometry. We then provide a brief introduction to toric varieties and toric vector bundles, including the construction of the fan of the projectivization of a toric bundle. This construction is used in Section~\ref{sec HK} to compute relevant geometric invariants of Hirzebruch--Kleinschmidt varieties, including a description of all big line bundles. In Section~\ref{sec hermitian vector bundles} we revisit the theory of Hermitian vector bundles over arithmetic curves and the notion of Arakelov degree. Additionally, we state the Poisson--Riemann--Roch formula, and demonstrate several estimates on the number of non-zero sections of a Hermitian vector bundle that are key in our proofs. In Section~\ref{sec: height zeta function projective} we define a ``standard height function'' on the set of rational points of the projective space $\PP^n$, in terms of the Arakelov degree of tautological Hermitian line bundles. With these ideas in mind, we revisit Maruyama's proof of Schanuel's estimate on the number of rational points of bounded standard height on projective spaces (see Corollary \ref{cor:SchanuelTheorem}). Finally, in Section~\ref{section:rational points}, we proceed to state and prove our main results on the asymptotic growth of the number of rational points of bounded height on Hirzebruch--Kleinschmidt varieties, with respect to height functions associated to big line bundles (Theorems~\ref{thm:generalVersionTheorem} and~\ref{thm:degenerate_case_general_theorem}), and in Section~\ref{subsection hirzebruch surfaces} we apply these theorems to Hirzebruch surfaces in order to further illustrate the scope of our results. 

\section*{Acknowledgments}

The second author thanks \textsc{Dan Loughran} for a warm reception during his academic stay at the University of Bath, and for useful discussions and comments. The authors thank \textsc{Giancarlo Lucchini-Arteche} for helpful comments on the first Galois cohomology group of an algebraic variety.

S.~Herrero's research is supported by ANID FONDECYT Iniciaci\'on grant 11220567 and by SNF grant CRSK-2{\_}220746. T.~Martinez's research is supported by ANID FONDECYT Iniciaci\'on grant 11220567, CONICYT-PFCHA Doctorado Nacional 2020-21201321, and by UTFSM/DPP Programa de Incentivo a la Investigación Cient\'ifica (PIIC) Convenio N°026/2022. P.~Montero's research is supported by ANID FONDECYT Regular grants 1231214 and 1240101.

 \section{Preliminaries}\label{sec preliminaries}

 \subsection{Basic notation}\label{subsection:Notation}
	Throughout this article we let~$K$ denote a number field of degree~$n_K$ over~$\mathbb{Q}$. Associated to~$K$ we have the following objects:
	\begin{itemize}
		\item The ring of integers~$\mathcal{O}_K$ and the associated arithmetic curve~$S:=\operatorname{Spec}(\CO_K)$.
		\item The number~$w_K$ of roots of unity in~$K$.
		\item The discriminant~$\Delta_K$, regulator~$R_K$ and class number~$h_K$.
		\item The set of discrete valuations~$\Val_f(K)$, which is in bijection with the set of non-zero prime ideals~$\pp\subset \mathcal{O}_K$.
        \item The set~$\Sigma_K$ of field embeddings~$K\hookrightarrow \CC$, and $r_1,r_2$ the number of real and complex Archimedean places, respectively. In particular, $\#\Sigma_K=n_K=r_1+2r_2$.
		\item Given~$v\in \Val_f(K)$ we denote by~$K_v$ the corresponding completion, and for~$x\in K_v$ we put~$|x|_v:=|\mathrm{Nr}_{K_v|\QQ_p}(x)|_p$ where~$p$ is the unique prime associated to the restriction of~$v$ to~$\mathbb{Q}$, and~$|\cdot|_p$ denotes the standard $p$-adic norm (namely, with~$|p|_p=p^{-1}$). Similarly, given~$\sigma\in \Sigma_K$ we denote by~$K_\sigma$ the completion of~$K$ with respect to the norm~$|x|_{\sigma}:=|\sigma(x)|$ where~$|\cdot |$ stands for the usual Euclidean norm on~$\CC$ (namely, $|x|=\sqrt{x\overline{x}}$). 
  With this normalization, the product formula
  $$\prod_{v\in \Val_f(K)}|x|_v \prod_{\sigma \in \Sigma_K}|x|_\sigma=1$$
  holds for every~$x\in K, x\neq 0$.
		\item $\eta: \Spec(\CO_K)\to \Spec(\ZZ)$ is the morphism of schemes induced by the inclusion~$\ZZ\hookrightarrow \mathcal{O}_K$.
        \end{itemize}

For a vector~$\mathbf{a}=(a_1,\ldots,a_r)\in \mathbb{N}_0^r$, we write $|\mathbf{a}|:=\sum_{i=1}^r a_i$.

In this article all varieties will be assumed to be irreducible, reduced and separated schemes of finite type over the base number field $K$. For simplicity we write~$\mathbb{A}^n$ and~$\mathbb{P}^{n}$ for the affine and projective space of dimension~$n$ over~$K$, and products of varieties are to be understood as fiber products over~$\mathrm{Spec}(K)$. The sheaf of regular functions on a variety~$X$ is denoted by~$\mathscr{O}_X$.

For a $K$-vector space~$V$, we denote by~$\mathbb{P}(V)$ the projective space of lines in~$V$. 
        
\subsection{A Tauberian theorem}

As mentioned in the Introduction, our main results 
follow from the analytic properties of certain height zeta functions, by using a  Tauberian theorem. We will use the following formulation, which follows from~\cite[Th\'eor\`eme~III]{DelangeTauberianThm}.
 
 \begin{thm}
 	Let $X$ be a countable set, $H:X\to \RR^{+}$ a function, and suppose that
 	$$\Z(s):=\sum_{x\in X} 
 	H(x)^{-s}$$
 	is absolutely convergent for $\Re(s)>a>0$ and
 	$$\Z(s)=\frac{g(s)}{(s-a)^b},$$ where $b$ is a positive integer and $g(s)$ is a holomorphic function in the half-plane $\Re(s)>a-\varepsilon$, for some~$\varepsilon>0$, with~$g(a)\neq 0$. Then, for every~$B>0$ the cardinality
 	$$N(X,H,B):=\#\{x\in X:H(x)\leq B\}$$
 	is finite, and
 	$$N(X,H,B)\sim \frac{g(a)}{(b-1)!\, a}B^a (\log B)^{b-1} \quad \text{as }B\to \infty.$$
 	\label{thm:tauberianThm}
 \end{thm}
 
 \subsection{Toric varieties}
\label{subsec:ToricVarieties}
We refer the reader to \cite{CoxLittleSchenckToricVarieties} for the general theory of toric varieties. Here we fix the notation that will be used along the article.
 
	Let $N\simeq \mathbb{Z}^d$ be a rank~$d$ lattice and~$M=\Hom_\ZZ(N, \ZZ)$ its dual lattice. Let us denote by~$\TT=\Spec(K[M])\simeq \GG_m^d$ the corresponding split algebraic torus, where~$K[M]$ is the~$K$-algebra generated by~$M$ as a semigroup. We identify the lattice $M$ with the group of characters of the torus~$\TT$ and $N$ with the one-parameter subgroups of~$\TT$.

Let $\Sigma$ be a fan in~$N_\RR:=N\otimes \RR$. This is, $\Sigma$ is a finite collection of strongly convex, rational polyhedral
cones $\sigma\subset N_\RR$ containing all the faces of its elements, and such that for every~$\sigma_1,\sigma_2\in \Sigma$, the intersection~$\sigma_1\cap \sigma_2$ is a face of both~$\sigma_1$ and~$\sigma_2$ (hence it is also in~$\Sigma$). We denote by $\Sigma(1)$ the set of rays (i.e., one-dimensional cones) in~$\Sigma$. More generally, for $\sigma\in \Sigma$ we denote by $\sigma(1)=\sigma \cap \Sigma(1)$ the set of rays on $\sigma$ and, by abuse of notation, we identify  rays with their primitive generators, i.e., with the unique primitive element~$u_\rho \in N$ that generates the ray~$\rho\in \Sigma(1)$. Also, given vector~$v_1,\ldots,v_r\in N_\RR$ we denote by~$\operatorname{cone}(v_1,\ldots,v_r)$ the cone that they generate.

Given a cone $\sigma\in \Sigma$, its dual $\sigma^\vee:=\{m\in M:\langle m, n\rangle \geq 0 \textup{ for all } n\in \sigma\}$ is a cone in $M_\RR$ and $U_\sigma=\Spec(K[\sigma^\vee \cap M])$ is the associated affine toric variety. The toric variety $X_{\Sigma}$ associated to~$\Sigma$ is obtained by gluing the affine toric varieties $\{U_\sigma\}_{\sigma\in  \Sigma}$ along~$U_{\sigma_1}\cap U_{\sigma_2}\simeq U_{\sigma_1\cap \sigma_2}$. It is a normal and separated variety that contains a maximal torus $U_{\{0\}}\simeq \TT$ as an open subset and admits an effective regular action of the torus $\TT$ extending the natural action of the torus over itself.  The toric variety~$X_{\Sigma}$ is smooth if and only if~$\Sigma$ is regular, meaning that every cone in~$\Sigma$ is generated by vectors that are part of a basis of~$N$. 


On a toric variety~$X_{\Sigma}$, each ray~$\rho\in \Sigma(1)$ corresponds to a prime $\TT$-invariant Weil divisor $D_\rho$, and the classes of~$D_\rho$ with~$\rho \in \Sigma(1)$ generate the class group~$\operatorname{Cl}(X_{\Sigma})$. In particular, every Weil divisor $D$ on $X$ is linearly equivalent to $\sum_{\rho\in \Sigma(1)} a_\rho D_\rho$ for some integers $a_\rho \in \mathbb{Z}$. Similarly, the classes of the $\TT$-invariant Cartier divisor generate the Picard group~$\operatorname{Pic}(X_{\Sigma})$. If the fan~$\Sigma$ contains a cone of maximal dimension~$d=\mathrm{dim}_{\RR}(N_\RR)$, then~$\operatorname{Pic}(X_{\Sigma})$ is a free abelian group of rank~$\#\Sigma(1)-d$.

The relevant toric varieties appearing in this paper are all smooth. We recall that on smooth varieties every Weil divisor is Cartier, and in particular~$\operatorname{Cl}(X_{\Sigma})\simeq \operatorname{Pic}(X_{\Sigma})$.

\subsection{Toric vector bundles}
\label{subsec:ToricVectorBundles}
Let $X$ be an algebraic variety. A variety $V$ is a vector bundle of rank $r$ over $X$ if there is a morphism $\pi: V\rightarrow X$ and an open cover $\{U_i\}_{i\in I}$ of $X$ such that:
\begin{enumerate}
    \item For every $i\in I$, there exists an isomorphism
    $$\varphi_i:\pi^{-1}(U_i)\xrightarrow{\sim}U_i\times \mathbb{A}^r,$$
    such that $p_1\circ \varphi_i=\pi|_{\pi^{-1}(U_i)}$, where $p_1$ denotes the projection onto the first coordinate.
    \item For every pair $i,j\in I$ there exists $g_{ij}\in \GL_r(\mathscr{O}_X(U_i\cap U_j) )$ such that the following diagram is commutative:
\begin{center}
\begin{tikzcd}
&&(U_i\cap U_j)\times\mathbb{A}^r                              \\
\pi^{-1}(U_i\cap U_j)  \arrow[rrd, "\varphi_j|_{\pi^{-1}(U_i\cap U_j)}"'] \arrow[rru, "\varphi_i|_{\pi^{-1}(U_i\cap U_j)}"] &  &                    \\ &  & (U_i\cap U_j)\times \mathbb{A}^r \arrow[uu, "\operatorname{Id}\times g_{ij}"']
\end{tikzcd}
\end{center}
\end{enumerate}
 The data $\{(U_i,\varphi_i)\}_{i\in I}$ satisfying (1) and (2) is called a trivialization for $\pi:V\to X$. The $g_{ij}$ are called transition matrices, and $V_p:=\pi^{-1}(p)\simeq \mathbb{A}^r$ is called the fiber at $p\in X$.

 Let $\pi:V\to X$ be a vector bundle of rank $r$ and $\{(U_i,\phi_i)\}_{i\in I}$ be a trivialization with transition matrices $g_{ij}$. Then, the functions $\operatorname{Id}\times g_{ij}$ induce isomorphisms
 $$\operatorname{Id}\times \overline{g}_{ij}:(U_i\cap U_j)\times \PP^{r-1}\xrightarrow{\sim} (U_i\cap U_j)\times \PP^{r-1},$$
 where $\overline{g}_{ij}\in \operatorname{PGL}_r(\mathscr{O}_X(U_i\cap U_j))$ is the projective map induced by~$g_{ij}$. This gives gluing data for a variety $\PP(V)$ and $\pi$ induces a morphism
$$\overline{\pi}:\PP(V)\rightarrow X,$$
with trivializations $\{(U_i,\overline{\varphi}_i)\}_{i\in I}$ for $\PP(V)$ where $$\overline{\varphi}_i:\overline{\pi}^{-1}(U_i)\xrightarrow{\sim}  U_i\times \PP^{r-1}.$$ 
The algebraic variety~$\PP(V)$ constructed in this way is called the projective bundle associated to $V$.

If $\mathcal{E}$ is a locally free sheaf over~$X$ of rank $r$, then $\mathcal{E}$ is the sheaf of sections of a vector bundle $\pi_\mathcal{E}:V_\mathcal{E}\to X$ of rank $r$. In this case, we define the projectivization of $\mathcal{E}$ to be
\begin{equation*}
		\PP(\mathcal{E}):=\PP(V_\mathcal{E}^\vee),
		\label{eq:defProjectiveBundle}
	\end{equation*}
 where~$V_\mathcal{E}^\vee$ denotes the vector bundle dual to~$V_\mathcal{E}$. The projective bundle~$\PP(\mathcal{E})$ has fiber over $p\in X$ given by $\PP(\mathcal{E})_p=\mathbb{P}(W_p)$ where $W_p=(V_\mathcal{E}^\vee)_p$. 


	Let $X$ be a toric variety defined by a fan $\Sigma$ in $N_\mathbb{R}$ as in Section~\ref{subsec:ToricVarieties}. A toric vector bundle over $X$ is a vector bundle $\pi:V\rightarrow X$ such that the  action of $\TT$ on $X$ extends to an action on $V$ in such a way that $\pi$ is $\TT$-equivariant and the action is linear on the fibers. The algebraic variety $V$ is not toric in general, and Oda \cite[Section~7.6]{Oda78} notes that the toric vector bundles which are toric varieties are precisely the decomposables ones, i.e., those of the form
	$V_{\mathcal{E}}$ with $\mathcal{E}=\CO_X(D_0)\oplus \cdots\oplus\CO_X(D_r)$ for some $\TT$-invariant Cartier divisors $D_0,\ldots, D_r$ on $X$.
 
	For decomposables toric vector bundles $V_{\mathcal{E}}\to X$ of rank $r+1$, we can construct the fan that defines the projective bundle $\mathbb{P}(\mathcal{E})\to X$ as a toric variety following \cite[\S 7.3]{CoxLittleSchenckToricVarieties}: Let $D_0,\ldots, D_r$ be $\TT$-invariant Cartier divisors, and write $D_i=\sum_{\rho\in \Sigma(1)}a_{i\rho}D_\rho$ with~$a_{i\rho}\in \ZZ$ for $i\in \{0,\ldots,r\}$. To construct the fan of $V_\mathcal{E}$ 
 we work in the vector space $N_\RR \oplus\RR^{r+1}$. We will denote by~$\{e_0,\ldots,e_r\}$ the canonical basis of $\RR^{r+1}\subset N_\RR \oplus \RR^{r+1} $ and write the elements of $N_\RR\oplus\RR ^{r+1}$ in the form $u+\lambda_0e_0+\cdots+\lambda_r e_r$, with $u\in N_\mathbb{R}$ and $\lambda_0,\ldots,\lambda_r\in \mathbb{R}$. Then, given $\sigma\in \Sigma$ we define $\overline{\sigma}\subset N_\mathbb{R}\oplus \mathbb{R}^{r+1}$ to be the Minkowski sum
	\begin{equation*}
			\overline{\sigma}
			:=\operatorname{cone}(u_\rho-a_{0\rho}e_0-\cdots-a_{r\rho}e_r:\rho\in\sigma(1))+\operatorname{cone}(e_0,\ldots,e_r),
	\end{equation*}
	where $u_\rho\in N$ is the primitive generator of the ray $\rho\in \sigma(1)$. The set of cones~$\{\overline{\sigma}\}_{\sigma \in \Sigma}$, together with their faces, defines a fan~$\overline{\Sigma}$ in~$N_\RR\oplus \RR^{r+1}$ for a toric variety $X_{\overline{\Sigma}}$ with a vector bundle structure $X_{\overline{\Sigma}}\to X_{\Sigma}$ whose sheaf of sections is isomorphic to $\mathcal{E}$, hence~$X_{\overline{\Sigma}}\simeq V_{\mathcal{E}}.$
 
Now consider $\PP(\mathcal{E})\rightarrow X_\Sigma$, which is a projective bundle with fibers isomorphic to $\PP^r$. To construct the fan of $\PP(\mathcal{E})$ we need to consider the dual sheaf $\mathcal{E}^\vee=\CO_X(-D_0)\oplus \cdots\oplus\CO_X(-D_r)$ and the associated vector bundle $V_{\mathcal{E}^\vee}=V_{\mathcal{E}}^{\vee}$. By the above construction, the fan of $V_{\mathcal{E}^\vee}$ is built from the cones 
$$\Cone(u_\rho+a_{0\rho}e_0+\cdots+a_{r\rho}e_r:\rho\in\sigma(1))+\Cone(e_0,\ldots,e_r),$$
	and their faces, where~$\sigma$ ranges over all the cones $\sigma\in \Sigma$. The fan of $\PP(\mathcal{E})$ is obtained as follows: for each $\sigma\in \Sigma$ and~$i\in \{0,\ldots,r\}$ we put $F_i=\Cone(e_0,\ldots,\hat{e}_i,\ldots,e_r)$, where~$\hat{e}_i$ means that we omit the vector~$e_i$, and define 
	$$\overline{\sigma}_i:=\Cone(u_\rho+a_{0\rho}e_0+\cdots+a_{r\rho}e_r:\rho\in\sigma(1))+F_i\subset N_\RR\times \RR^{r+1}.$$
	Let $\sigma_i$ be the image of $\overline{\sigma}_i$ under the canonical projection $N_\RR\oplus \RR^{r+1}\to N_\RR\oplus \overline{N}_\RR$, where $\overline{N}_\RR:=\RR^{r+1}/\RR(e_0+e_1+\cdots+e_r)$. Then, we have the following result (see~\cite[Proposition~7.3.3]{CoxLittleSchenckToricVarieties}).
 
	\begin{prop}\label{prop cone of projective bundle}
		The cones $\{\sigma_i\}_{\sigma\in \Sigma,\; i\in \{0,\ldots,r\}}$ and their faces form a fan $\Sigma_\mathcal{E}$ in $N_\RR\oplus \overline{N}_\RR$ whose associated toric variety $X_{\Sigma_\mathcal{E}}$ is isomorphic to $\PP(\mathcal{E})$.
	\end{prop}
	
	In practice, we will replace $\overline{N}_\RR=\RR^{r+1}/\RR(e_0+e_1+\cdots+e_r) $ by $\RR^r$ with basis $e_1,\ldots, e_r$ and define $e_0:=-e_1-\cdots-e_r$. Hence, we consider~$F_i=\Cone(e_0,\ldots,\hat{e}_i,\ldots,e_r)\subset \RR^r$, and for a cone $\sigma\in \Sigma$ we get 
	$$\sigma_i=\Cone(u_\rho+(a_{1\rho}-a_{0\rho})e_1+\cdots+(a_{r\rho}-a_{0\rho})e_r:\rho\in \sigma(1))+F_i\subset N_\RR\oplus \RR ^r.$$
	The cones $\sigma_i$ and their faces define a fan $\Sigma_\mathcal{E}$ in $N_\RR\oplus \RR ^r$ for $\PP(\mathcal{E})$.

\section{Hirzebruch--Kleinschmidt varieties}\label{sec HK}
 
	Given integers $r\geq 1, t\geq 2, 0\leq a_1\leq \ldots\leq a_r,$ consider the vector bundle
	$$\mathcal{E}:=\CO_{\PP ^{t-1}}\oplus\CO_{\PP ^{t-1}}(a_1)\oplus\cdots\oplus\CO_{\PP ^{t-1}}(a_r).$$
	Here, as usual, we put $\CO_{\PP ^{t-1}}(a_i):=\CO_{\PP ^{t-1}}(a_iH_0)$ where $H_0\subset \PP^{t-1}$ is a hyperplane, which we can choose as~$H_0=\{x_0=0\}$ in homogeneous coordinates~$[x_0:\ldots:x_t]$.

We can use Proposition~\ref{prop cone of projective bundle} to describe the fan $\Sigma_\mathcal{E}$ of the smooth toric variety~$\PP(\CE)$. For this, it is enough to describe the maximal cones in~$\Sigma_\mathcal{E}$. Consider $N\oplus \overline{N}=\ZZ^{t-1}\oplus \ZZ^r$ with canonical bases~$\{u_1,\ldots,u_{t-1}\}$ and~$\{e_1,\ldots,e_r\}$ for $N$ and $\overline{N}$, respectively. As before, we set $u_0:=-u_1-\cdots-u_{t-1}$ and $e_0:=-e_1-\cdots-e_r$. Note that~$\{u_0,\ldots u_{t-1}\}$ is the set of primitive generators of the rays of a fan for the toric variety~$\PP^{t-1}$ with $u_0$ corresponding to the divisor $H_0$. 
 
Put~$a_0:=0$ and~$D_i:=a_iH_0$ for~$i\in \{0,\ldots ,r\}$. As in Section~\ref{subsec:ToricVectorBundles}, we write~$D_i=\sum_{\rho\in \Sigma(1)}a_{i\rho}D_\rho$. For~$\rho\in \Sigma(1)$ with~$u_{\rho}=u_i$ we define
$$v_i:=u_\rho+(a_{1\rho}-a_{0\rho})e_1+\ldots+(a_{r\rho}-a_{0\rho})e_r=
\left\{
\begin{array}{ll}
u_0+a_1e_1+\cdots+a_re_r   & \text{ if }i=0, \\
 u_i    & \text{ if }i\in \{1,\ldots,t-1\}.\
\end{array}
\right.$$

Since the maximal cones of $\PP^{t-1}$ are $\{\Cone(u_0,\ldots,\hat{u}_i,\ldots,u_{t-1})\}_{i\in \{0,\ldots,t-1\}}$, we see that the maximal cones of $\Sigma_\CE$ are 
	$$\Cone(v_0,\ldots,\hat{v}_j,\ldots,v_{t-1})+\Cone(e_0,\ldots,\hat{e}_i,\ldots,e_r),$$
	for $j\in \{0,\ldots,t-1\}$ and $i\in \{0,\ldots, r\}$. Therefore, the primitive generators of the rays of $\Sigma_\mathcal{E}$ are $v_0,\ldots,v_{t-1}, e_0, e_1,\ldots,e_r$ (compare with~\cite[p.~256]{Kleinschmidt88}).

    Note that $\mathbb{P}(\mathcal{E})$ has dimension  
    $$d:=\dim(\mathbb{P}(\mathcal{E}))=r+t-1.$$
    Since~$\Sigma_\mathcal{E}$ contains cones of maximal dimension~$d$ and $\#\Sigma_\mathcal{E}=d+2$, we conclude that $\operatorname{Pic}(\mathbb{P}(\mathcal{E}))\simeq \mathbb{Z}^2$. Conversely, Kleinschmidt \cite{Kleinschmidt88} proved the following classification result of smooth projective\footnote{Actually, the classification in \cite{Kleinschmidt88} does not assume the projectivity hypothesis.} varieties of Picard rank 2.
    
	\begin{thm}[Kleinschmidt]
		Let $X_\Sigma$ be a smooth projective toric variety with $\Pic(X_\Sigma)\simeq \ZZ^2$. Then, there exists integers $r\geq 1, t\geq 2, 0\leq a_1\leq\cdots\leq a_r$ with $r+t-1=d=\dim(X_\Sigma)$ such that
		$$X_\Sigma\simeq \PP(\CO_{\PP ^{t-1}}\oplus\CO_{\PP ^{t-1}}(a_1)\oplus\cdots\oplus\CO_{\PP ^{t-1}}(a_r)).$$
	\end{thm}

We now recall the definition of Hirzebruch--Kleinschmidt variety given in the Introduction, which is the main family of algebraic varieties studied in this paper.

\begin{defi}\label{defi:HK varieties}
    Given integers $r\geq 1$, $t\geq 2$ and $0\leq a_1 \leq \cdots \leq a_r$, the  \emph{Hirzebruch--Kleinschmidt variety} $X_d(a_1,\ldots,a_r)$ is defined as
    \[
    X_d(a_1,\ldots,a_r):=\PP ( \CO_{\PP^{t-1}} \oplus \CO_{\PP ^{t-1}}(-a_r) \oplus \CO_{\PP ^{t-1}}(a_1-a_r)\oplus\cdots\oplus\CO_{\PP ^{t-1}}(a_{r-1}-a_r)),
    \]
    where $d=\dim(X_d(a_1,\ldots,a_r))=r+t-1$. We denote by $\pi:X_d(a_1,\ldots,a_r)\to \mathbb{P}^{t-1}$ the associated projective bundle.
    \label{def:HKVarieties}
\end{defi}

\begin{remark}\label{rmk isom projective bundles}
    Recall that for any line bundle $\mathcal{L}\in \operatorname{Pic}(X)$ and every locally free sheaf~$\mathcal{E}$ on an algebraic variety $X$ there is a canonical isomorphism of projective bundles~$\mathbb{P}(\mathcal{E}\otimes_{\mathcal{O}_X} \mathcal{L})\simeq \mathbb{P}(\mathcal{E})$ (see e.g.~\cite[Lemma~7.0.8(b)]{CoxLittleSchenckToricVarieties}). In particular, 
    $$X_d(a_1,\ldots,a_r)\simeq \PP(\CO_{\PP ^{t-1}}\oplus\CO_{\PP ^{t-1}}(a_1)\oplus\cdots\oplus\CO_{\PP ^{t-1}}(a_r)).$$
    We choose the description in Definition \ref{def:HKVarieties} because it allows for a simpler characterization of the cone of effective divisors (see Proposition~\ref{prop:EffectiveCone} below).
\end{remark}

Note that the minimal generators of the rays in the fan of~$X_d(a_1,\ldots,a_r)$ are the vectors 
$$w_i:=
\left\{
\begin{array}{ll}
u_0-a_r e_1+(a_1-a_r)e_2+\ldots+(a_{r-1}-a_r)e_r  & \text{ if }i=0, \\
 u_i    & \text{ if }i\in \{1,\ldots,t-1\},
\end{array}
\right.$$
 together with the primitive elements $e_0, \ldots, e_r$. From now on we denote by $D_i$ the divisor on $X_d(a_1,\ldots,a_r)$ corresponding to the minimal generator $w_i$, for $i\in \{0, \ldots, t-1\}$, and by $E_j$ the corresponding divisor corresponding to $e_j$, for $j\in \{0,\ldots,r\}$. It is easy to see that~$\CO_X(D_i)\simeq \pi^*\CO_{\PP^{t-1}}(1)$ for every~$i\in \{1, \ldots, t-1\}$ and~$\CO_X(E_0)\simeq \CO_{X}(1)$ (e.g., by using local trivializations). 
  


\subsection{Effective divisors}\label{subsection effective divisors}

The following description of the cone of effective divisors on~$X_d(a_1,\ldots,a_r)$ is a natural extension to higher dimensions of the description for Hirzebruch surfaces (see e.g. \cite[Chapter V, Corollary~2.18]{Har77}). 

\begin{prop}\label{prop effective cone}
		Let $X=X_d(a_1, \ldots,a_r)$ be a Hirzebruch--Kleinschmidt variety and let us denote by~$f$ the class of~$\pi^\ast \mathcal{O}_{\mathbb{P}^{t-1}}(1)$ and by~$h$ the class of~$\mathcal{O}_X(1)$, both in~$\Pic(X)$. 
  Then:
  \begin{enumerate}
  \item $\Pic(X)\simeq \ZZ h \oplus \ZZ f$.
      \item The anticanonical divisor class of $X$ is given by
      $$-K_X=(r+1)h+\left((r+1)a_r+ t-|\mathbf{a}| \right) f,$$
      where $|\mathbf{a}|=\sum_{i=1}^r a_i$.
      \item The cone of effective divisors of $X$ is given by
      $$\Lambda_{\textup{eff}}(X)=\{\lambda h+\mu  f:\lambda \geq 0, \mu \geq 0\}\subset \Pic(X)_\RR$$
      where $\Pic(X)_\RR:=\Pic(X)\otimes_{\ZZ}\RR$.
  \end{enumerate}
  \label{prop:EffectiveCone}
	\end{prop}

\begin{proof} We follow \cite[Section~7]{HuangRationalAppToricVar2021}.
With the above notation we have that the vectors $w_1,\ldots,w_{t-1},e_1,\ldots,e_r$ form a basis for $\ZZ^{t-1}\oplus\ZZ ^r$. For $1\leq i\leq t-1, 1\leq j\leq r$ we denote by $w_i^*,e_j^*$ the corresponding dual basis elements. We then compute the  divisors of the characters $\chi^{w_i^*}$ for $i\in \{1,\ldots,t-1\}$, which are
\[
\textup{div}\left( \chi^{w_i^*}\right) =\sum_{k=0}^{t-1}\left\langle w_i^*,w_k \right\rangle D_k+\sum_{k=0}^r\left\langle w_i^*,e_k \right\rangle E_k =-D_0+D_i,
\]
	and similarly the divisors of the characters $\chi^{e_j^*}$ for $j\in\{1,\ldots,r\}$ are
 \[
 \textup{div}\left( \chi^{e_j^*}\right)=(a_{j-1}-a_r)D_0-E_0+E_j
 \]
 (recalling that~$a_0:=0$). Therefore, in $\Pic(X_d(a_1,\ldots,a_r))$ we have the relations 
 \begin{equation}\label{eq relations D_i and E_j}
     D_i=D_0 \text{ and }E_j=E_0+(a_r-a_{j-1})D_0 \text{ for $1\leq i\leq t-1, 1\leq j\leq r$.}
 \end{equation}
 In particular, $\Pic(X_d(a_1,\ldots,a_r))=\ZZ\cdot E_0\oplus \ZZ\cdot D_0=\ZZ h \oplus \ZZ f$. This proves item~$(1)$.
 
	It follows from~\cite[Theorem~8.2.3]{CoxLittleSchenckToricVarieties} that the anticanonical divisor class of $X=X_d(a_1,\ldots,a_r)$ is given by the class
 \[
 \sum_{i=0}^{t-1}D_i+\sum_{j=0}^rE_j=(r+1)E_0+\left( (r+1)a_r+t-|\mathbf{a}|\right)D_0.
 \]
 This proves item~$(2)$. 

Finally, by~\cite[Lemma~15.1.8]{CoxLittleSchenckToricVarieties} the effective cone equals the cone generated by the classes of the divisors~$D_i$ and~$E_j$. Hence, item~$(3)$ is a consequence of~\eqref{eq relations D_i and E_j}. This completes the proof of the proposition.

\end{proof}

It follows from~\cite[Theorem~2.2.26]{LazI} that a divisor class in~$\Pic(X_d(a_1,\ldots,a_r))$ is  big if and only if it lies in the interior of the effective cone~$\Lambda_{\textup{eff}}(X)$. Hence, we get the following corollary from Proposition~\ref{prop effective cone}.

\begin{cor}\label{cor big divisors}
  Let~$L=\lambda h+\mu f$ with~$\lambda,\mu\in \ZZ$, where~$\{h,f\}$ is the basis of~$\Pic(X_d(a_1,\ldots,a_r))$ given in Proposition~\ref{prop effective cone}. Then, $L$ is big if and only if~$\lambda>0$ and~$\mu>0$.
\end{cor}

In particular, Proposition~\ref{prop effective cone} implies that the anticanonical divisor class in a Hirzebruch--Kleinschmidt variety is big. This is true for any smooth projective toric variety.


\subsection{Peyre's $\alpha$-constant}\label{subsection Peyre's alpha}

Let us recall that \emph{Peyre's $\alpha$-constant} of an almost Fano\footnote{Every smooth projective toric variety is almost Fano.} variety~$X$ is defined as
\[
\alpha(X)=\frac{1}{\left(\rk \Pic(X)-1\right)!}\int_{\Lambda_{\textup{eff}}(X)^\vee} e^{-\left \langle -K_X,\mathbf{y}\right \rangle }\textup{d}\mathbf{y},
\]
where~$\langle \cdot ,\cdot \rangle$ denotes the natural pairing~$\Pic(X)_\RR \times \Pic(X)_\RR^\vee \to \RR$ and $\textup{d}\mathbf{y}$ denotes the Lebesgue measure on~$\Pic(X)_\RR^\vee $ normalized to give covolume 1 to the lattice~$\Pic(X)^\vee$ (see e.g.~\cite[Definition~2.5]{Pey/Tsch2001}). 

A straightforward computation using Proposition~\ref{prop effective cone} give us the following result.

 \begin{lem}
 \label{lem:calculusAlpha}
     Let $X=X_d(a_1,\ldots,a_r)$ be a Hirzebruch--Kleinschmidt variety. Then, its $\alpha$-constant is given by
     \[
     \alpha(X)=\int_{0}^{\infty}\int_{0}^{\infty} e^{-(r+1)y_1-((r+1)a_r+t-|\mathbf{a}|)y_2}\textup{d}y_1\, \textup{d}y_2 =\frac{1}{(r+1)\left((r+1)a_r+t- |\mathbf{a}|\right) }.
     \]
 \end{lem}

\subsection{Restriction of big line bundles}\label{sec restriction of divisors}

Given integers $r\geq 1$, $t\geq 2$ and $0\leq a_1 \leq \cdots \leq a_r$ we defined 
the Hirzebruch--Kleinschmidt variety~$X=X_d(a_1,\ldots,a_r)$ of dimension $d=r+t-1$ as the projective bundle~$\PP(\mathcal{E})$ where
$$\mathcal{E}:=\PP ( \CO_{\PP^{t-1}} \oplus \CO_{\PP ^{t-1}}(-a_r) \oplus \CO_{\PP ^{t-1}}(a_1-a_r)\oplus\cdots\oplus\CO_{\PP ^{t-1}}(a_{r-1}-a_r)).$$


\begin{defi}\label{def F}
Put
$$\mathscr{Y}:=\CO_{\PP^{t-1}}(-a_r)\oplus \CO_{\PP^{t-1}}(a_1-a_r)\oplus \cdots\oplus \CO_{\PP^{t-1}}(a_{r-1}-a_r)$$
and define, as in the Introduction, the projective subbundle $F:=\mathbb{P}(\mathscr{Y})\subset X_d(a_1,\ldots,a_r)$.
\end{defi}

Note that, when~$r\geq 2$ we have (see Remark~\ref{rmk isom projective bundles})
\begin{equation}\label{eq isom F with X_(d-1)}
F\simeq \mathbb{P}(\mathscr{Y}\otimes \mathcal{O}_{\mathbb{P}^{t-1}}(a_r-a_{r-1})) = X_{d-1}(a_1,\ldots,a_{r-1}),
\end{equation}
hence~$F$ is a Hirzebruch--Kleinschmidt variety of dimension~$d-1$, while in the case~$r=1$ we have that~$F\simeq \PP^{t-1}$ is a projective space.

Denote by~$\iota:F\to X$ the inclusion map. Given a class~$L\in \Pic(X)$, we denote by~$L|_F:=\iota^{*}L$ its restriction to~$F$. 

In this section we prove the following results concerning the restriction to~$F$ of line bundles on~$X$.

 \begin{lem}
 \label{lem:restrictionBigDivisor}
 Assume~$r\geq 2$, let~$X=X_d(a_1,\ldots,a_r)$, $X'=X_d(a_1,\ldots,a_{r-1})$, and let~$\{h,f\}$, $\{h',f'\}$ be the bases of~$\Pic(X)$ and~$\Pic(X')$, respectively, given in Proposition \ref{prop:EffectiveCone}. If $L=\lambda h+\mu f\in \Pic(X)$, then $L|_F\in \Pic(F)$ corresponds under the canonical isomorphism~\eqref{eq isom F with X_(d-1)} to the class
 $$\lambda h'+(\mu-\lambda (a_r-a_{r-1}))f'\in \Pic(X').$$
 In particular, $L|_F$ is a big line bundle class in $\Pic(F)$ if and only if $\lambda>0$ and $\mu>\lambda (a_r-a_{r-1})$.
 \end{lem}
	\begin{proof}
 On the one hand, given a closed immersion~$\varphi:X\hookrightarrow \PP^N$ (for some~$N>0$) we have
 $$h|_F=\iota^*h=\iota^*(\CO_X(1))=\iota^* (\varphi^*(\CO_{\PP^N}(1)))= (\varphi\circ \iota)^*(\CO_{\PP^N}(1))=\CO_F(1).$$
 Now, under the isomorphism~\eqref{eq isom F with X_(d-1)}, the class $\CO_F(1)\in \Pic(F)$ corresponds to~$\CO_{X'}(1)\otimes (\pi')^*\CO_{\PP ^{t-1}}(a_{r-1}-a_r)$, where~$\pi':X'\to \PP^{t-1}$ is the projection map of~$X'$ (see~\cite[Chapter 2, Lemma~7.9]{Har77}). 
 This shows that~$h|_F$ corresponds to~$h'+(a_{r-1}-a_{r})f'$. 
On the other hand, since~\eqref{eq isom F with X_(d-1)} is an isomorphism of projective bundles over~$\PP^{t-1}$, we have that~$f|_F=(\pi|_F)^*(\CO_{\PP^{t-1}}(1))$ corresponds to~$(\pi')^*(\CO_{\PP^{t-1}}(1))=f'$. This implies that~$L|_F=\lambda h|_F+\mu f|_F$ corresponds to
$$\lambda (h'+(a_{r-1}-a_{r})f')+\mu f'=\lambda h'+(\mu-\lambda (a_r-a_{r-1}))f'.$$
Finally, the last statement follows from Corollary~\ref{cor big divisors}. This proves the lemma. 
 \end{proof}

 \begin{lem}
 \label{lem:restrictionBigDivisorII}
    Assume~$r=1$, i.e.~$X=X_d(a)$ with~$a\geq 0$ an integer. If $L=\lambda h+\mu f\in  \Pic(X)$, then~$L|_F$ corresponds under the isomorphism~$F\simeq \PP^{t-1}$ to the class of the line bundle~$\CO_{\PP^{t-1}}(\mu-a\lambda)\in \Pic(\PP^{t-1})$. In particular,~$L|_F$ is big if and only if~$\mu>a\lambda$.
\end{lem}
\begin{proof}
    The proof is similar to the case~$r\geq 2$, but now we use that~$h|_F=\pi^*(\CO_{\PP^{t-1}}(-a))$ corresponds to~$\CO_{\PP^{t-1}}(-a)$, while~$f|_F$ corresponds to~$\CO_{\PP^{t-1}}(1)$. We omit the details for brevity.
\end{proof}

\begin{remark}
    When applied to the anticanonical class~$-K_X=(r+1)h+((r+1)a_r+t-|\mathbf{a}|)$ (see Proposition~\ref{prop effective cone}(3)), Lemmas~\ref{lem:restrictionBigDivisor} and~\ref{lem:restrictionBigDivisorII} show that~$-K_X$ remains big when restricted to each component in the decomposition~\eqref{eq natural decomposition HK} if and only if~$t>|\mathbf{a}|$, and this is exactly the case when~$X_d(a_1,\ldots,a_r)$ is Fano according to~\cite[Theorem~2(2)]{Kleinschmidt88}.
\end{remark}
 
\section{Hermitian vector bundles over arithmetic curves}\label{sec hermitian vector bundles}

   In this section we  follow closely the presentation in \cite{BostTIELHVB}. Let us recall that $S=\operatorname{Spec}(\CO_K)$.
 
	\begin{defi}
		A \emph{Hermitian vector bundle} $\overline{E}$ over $S$ is a pair $(E,h)$ where $E$ is a finitely generated projective $\CO_K$-module and $h=\{h_\sigma\}_{\sigma\in \Sigma_K}$ is a family of positive definite  Hermitian forms over the family of complex vector spaces $\{E\otimes_{\CO_K,\sigma}\CC\}_{\sigma\in \Sigma_K}$, which are invariant under conjugation, i.e., $\|e\otimes_{\overline{\sigma}}\overline{\lambda}\|_{\overline{\sigma}}=\|e\otimes_{\sigma} \lambda\|_{\sigma}$, for all~$e\in E,\lambda \in \CC$, where~$\|\cdot \|_{\sigma}:=\sqrt{h_{\sigma}(\cdot,\cdot)}$ is the usual Hermitian norm associated to~$h_{\sigma}$. An element of $E$ is called a \emph{rational section}. 
	\end{defi}
	
	The rank of $\overline{E}=(E,h)$ is defined as the rank of $E$ as $\CO_K$-module, i.e., as the dimension of the complex vector spaces~$E\otimes_{\CO_K,\sigma}\CC$. A morphism between two Hermitian vector bundles $\overline{E}_1=(E_1,h_1)$ and $\overline{E}_2=(E_2,h_2)$ is a $\CO_K$-homomorphism $\varphi:E_1\to E_2$ such that~$\|\varphi(e\otimes_\sigma \lambda)\|_{2,\sigma}\leq \|e\otimes_\sigma\lambda\|_{1,\sigma}$ for all~$e\in E,\lambda\in \CC$ and~$\sigma\in \Sigma_K$.  
 An isomorphism of Hermitian vector bundles is a bijective morphism inducing an isometry~$E_1\otimes_{\CO_K,\sigma}\CC\to E_2\otimes_{\CO_K,\sigma}\CC$ for every~$\sigma$.
 

    	We denote by $\widehat{\Pic}(S)$ the set of \emph{Hermitian  line bundles} (i.e., Hermitian vector bundles of rank 1) over $S$ up to isomorphism. Note that a Hermitian line bundle is uniquely determined by the underlying projective~$\CO_K$-module and the values~$\|1\|_\sigma$ with~$\sigma\in \Sigma_K$.

	\begin{example}
		Let $r\in \mathbb{N}_{\geq 1}$. The \emph{trivial Hermitian vector bundle} $\overline{\CO_K^{\oplus r}}:=(\CO_K^{\oplus r}, h)$ over $S$ is defined by considering for each $\sigma\in \Sigma_K$ the Hermitian form
		$$h_\sigma(a\otimes_\sigma \lambda_1,b\otimes_\sigma \lambda_2):=\lambda_1 \overline{\lambda_2}\langle \sigma(a), \overline{\sigma(b)}\rangle_{\CC^r},$$
		where for $a=(a_i)\in \CO_K^{\oplus r}$ we put $\sigma(a):=(\sigma(a_i))\in \CC^r$ and~$\langle \cdot, \cdot \rangle_{\CC^r}$ denotes the standard bilinear form in~$\mathbb{C}^r$. 
		\label{ex:trivialBundle}
	\end{example}

	\subsection{Operations with Hermitian vector bundles.}
It is possible to extend the usual constructions of linear algebra to Hermitian vector bundles. 
Here we present the ones that will be needed in this paper.
 
	Let $\overline{E}_1=(E_1,h_1)$ and $\overline{E}_2=(E_1,h_2)$ be Hermitian vector bundles over $S$. 
	\begin{itemize}
		
    		\item \textbf{Direct sum. }We define $\OL{E}_1\oplus \OL{E}_2$ as the pair $(E,h)$ where $E:=E_1\oplus E_2$, and over $$(E_1\oplus E_2)\otimes_{\CO_K,\sigma}\CC\simeq (E_1\otimes_{\CO_K,\sigma}\CC)\oplus( E_2\otimes_{\CO_K,\sigma}\CC)$$ we define 
      $$h_\sigma(((e_1\otimes_\sigma\lambda_1),(d_1\otimes_\sigma \mu_1)),((e_2\otimes_\sigma\lambda_2), (d_2\otimes_\sigma \mu_2))):=h_{1,\sigma}(e_1\otimes_\sigma\lambda_1,e_2\otimes_\sigma\lambda_2)+h_{2,\sigma}(d_1\otimes_\sigma \mu_1,d_2\otimes_\sigma \mu_2).$$
		We have $\operatorname{rk}(\OL{E}_1\oplus \OL{E}_2)=\operatorname{rk}(\OL E_1)+\operatorname{rk}(\OL E_2).$
		\item \textbf{Tensor product}. 
		Define $\OL E_1\otimes \OL E_2$ as the pair $(E,h)$, where $E:=E_1\otimes_{\CO_K}E_2$, and over
		$$(E_1\otimes_{\CO_K}E_2)\otimes_{\CO_K,\sigma} \CC\simeq (E_1\otimes_{\CO_K,\sigma} \CC) \otimes_\CC (E_2\otimes_{\CO_K,\sigma} \CC)$$
		we define
		$$h_\sigma((e_1\otimes_\sigma\lambda_1\otimes e_2\otimes_\sigma\lambda_2), (d_1\otimes_\sigma \mu_1\otimes d_2\otimes_\sigma \mu_2)):=h_{1,\sigma}(e_1\otimes_\sigma\lambda_1, d_1\otimes_\sigma \mu_1)h_{2,\sigma}(e_2\otimes_\sigma\lambda_2,d_2\otimes_\sigma \mu_2).$$
		Then, $\operatorname{rk}(\OL E_1\otimes \OL E_2)=\operatorname{rk}(\OL E_1)\operatorname{rk}(\OL E_2).$
		\item\textbf{ Dual}. Given a Hermitian vector space $(V,h)$, we can identify $V$ with its dual $V^\vee=\Hom_\CC(V, \CC)$ by means of the application $v\mapsto H_v$, where $H_v$ is the functional defined by $H_v(u)=h(u,v)$. With this identification, $V^\vee$ inherits a Hermitian structure given by $h_{V^\vee}(H_u, H_v):=\OL{h(u,v)}$.
  		We thus define $\OL E_1^\vee$ as the pair $(E,h)$ where $E=E_1^\vee=\Hom_{\CO_K}(E, \CO_K)$ and $h$ is the family of Hermitian forms defined in 
		\begin{equation*}
			\begin{split}
				\Hom_{\CO_K}(E,\CO_K)\otimes_{\CO_K,\sigma}\CC&\simeq \Hom_\CC(E\otimes_{\CO_K,\sigma}\CC,\CO_K\otimes_{\CO_K,\sigma}\CC)\\
				&\simeq \Hom_\CC(E\otimes_{\CO_K,\sigma}\CC,\CC)\\
				&= \left( E\otimes_{\CO_K,\sigma}\CC\right) ^\vee
			\end{split}
		\end{equation*}
		as we explained before for Hermitian vector spaces. In particular, $\operatorname{rk}(\OL E_1^\vee)=\operatorname{rk}(\OL E_1)$.
		\item \textbf{Alternating products}. Given $m\in \mathbb{N}_{\geq 1}$, define $\bigwedge^m \OL E_1$ as the pair $(E,h)$ where $E=\bigwedge^m E_1$ and $h$ is the family of Hermitian forms defined by
		$$h_\sigma(e_1\wedge \cdots\wedge e_m, d_1\wedge\cdots\wedge d_m):=\det(h_{1,\sigma}(e_i,d_j)).$$
		We have $\operatorname{rk}\left( \bigwedge^m \OL E_1\right) =\left(\begin{array}{c}
			\operatorname{rk}(E_1)	\\m
			
		\end{array} \right)$. In particular, the  \emph{determinant} $\det\left( \OL E_1\right):=\bigwedge^{\operatorname{rk}(E_1)}\OL E_1$ is a Hermitian line bundle. 
		\item \textbf{Direct image.} Recall that $\eta:\operatorname{Spec}(\CO_K)\to \operatorname{Spec}(\ZZ)$ is the morphism induced by the inclusion~$\ZZ\hookrightarrow \CO_K$. Given a Hermitian vector bundle $\overline{E_1}=(E_1, h_1)$ over $S$, we can define a Hermitian vector bundle $\eta_*\overline{E}=(E,h)$ over $\Spec(\ZZ)$ in the following way: 
		Consider $E=E_1$ but as a free $\ZZ$-module of rank $[K:\QQ]\cdot \rk E$, and note that
  \[
   E\otimes_\ZZ \CC =E_1\otimes_\ZZ \CC \simeq E_1\otimes_{\CO_K}\left(\CO_K\otimes_ \ZZ\CC \right) \simeq \bigoplus_{\sigma}(E_1\otimes_{\CO_K,\sigma}\CC).
  \]
		Then, given $a=(a_\sigma),b=(b_{\sigma})\in 
  \bigoplus_{\sigma}(E_1\otimes_{\CO_K,\sigma}\CC),$ we define $h(a,b):=\sum_\sigma h_{1,\sigma}(a_\sigma,b_\sigma)$.
	\end{itemize}
	\begin{remark}
		The set $\widehat{\Pic}(S)$ has a group structure induced by the tensor product. The inverse element is induced by the dual and the identity element is the class of the trivial Hermitian line bundle~$\overline{\CO_K}$ defined in Example \ref{ex:trivialBundle}.
	\end{remark}

	\begin{example}\label{ex:ConstO(1)}
		For an integer~$n\geq 1$ consider the trivial Hermitian vector bundle $\OL{\CO_K^{\oplus n+1}}$ (see Example \ref{ex:trivialBundle}) and the projective space of lines $\PP^n(K)=\PP(K^{\oplus n+1})$. Each line $\ell \subset K ^{\oplus n+1}$ defines a finitely generated projective $\CO_K$-module $\CO_K^{\oplus n+1}\cap \ell$ whose corresponding complexifications are metrized using the restriction of the ambient Hermitian forms. Then, for each 
  point $P=\ell \in \PP^n(K)$
  we get a Hermitian line bundle denoted by $\overline{\CO_{\PP^n}(-1)_P}$. Its dual is denoted by $\overline{\CO_{\PP^n}(1)_P}.$
  The metric on $\overline{\CO_{\PP^n}(1)_P}\otimes_{\CO_{K,\sigma}}\CC$ constructed in this way is the Fubini--Study metric (see e.g. \cite[Example 1.2.45]{LazPAGI}). As usual, by taking duals and tensor powers, we can define for every~$a\in \ZZ$ the Hermitian line bundle~$\overline{\CO_{\PP^n}(a)_P}$.
	\end{example}
	\begin{remark}
		Note the analogy between the construction in the example above and the classical construction of the tautological line bundle of $\PP ^n$. In particular, we can interpret
		$\overline{\CO_{\PP^n}(-1)_P}$ as $\CO_{\PP^n }(-1)_P\cap \overline{\CO_K^{\oplus n+1}}$ where $\CO_{\PP^n }(-1)$ is the tautological geometric line bundle over the projective space $\PP^n$.
		
	\end{remark}

	\subsection{Arakelov degree}

    In this section, we will review some of the important properties of the Arakelov degree of Hermitian vector bundles. 

	\begin{defi} Let $\overline{L}=(L,h)$ be a Hermitian line bundle over $S$ and $s\in L\setminus \{0\}$ a non-trivial rational section. The  \emph{Arakelov degree of the line bundle} $\overline{L}$ is defined as
		\begin{equation*}
			\begin{split}
				\widehat{\deg}(\overline{L}):&=\log|L/\CO_Ks|-\sum_{\sigma\in \Sigma_K} \log \|s\|_\sigma \\	&=\sum_{\pp\subset \CO_K}v_{\pp}(s)\log N(\pp)-\sum_{\sigma\in \Sigma_K} \log \|s\|_{\sigma},	
			\end{split}
		\end{equation*}
		where $\pp$ runs over all non-zero prime ideals of $\CO_K$, $N(\pp)=|\CO_K/\pp|$ is the norm of the ideal~$\pp$  and $v_\pp(s)$ denotes the $\pp$-adic valuation of $s$ seen as a section of the invertible sheaf over $S$ associated to~$L$. 
  More concretely, if we consider the localization $L_\pp:=L\otimes_{\CO_K}\CO_{K,\pp}$, then $L_\pp$ is a free $\CO_{K,\pp}$-module of rank one and therefore there exists an isomorphism (a trivialization) $i_\pp:L_\pp\xrightarrow{\sim} \CO_{K,\pp}$. Then,  $v_\pp(s)=v_\pp(i_\pp(s\otimes 1))$ via this identification. It follows from the product formula that the definition above is independent of the choice of the non-trivial section $s$.

The \emph{Arakelov degree of a Hermitian vector bundle} $\overline{E}=(E,h)$ over $S$ is defined as
		$$\widehat{\deg}(\overline{E}):=\hdeg(\det(\overline{E})),$$
		and its  \emph{norm} is defined as $\operatorname{N}(\overline{E}):=e^{\widehat{\operatorname{deg}}(\overline{E})}\in \mathbb{R}_{>0}$.
		\end{defi}

 
	\begin{example}
		For the trivial Hermitian vector bundle $\overline{\CO_K^{\oplus r}}=(\CO_K^{\oplus r}, h)$, we have $\hdeg(\overline{\CO_K^{\oplus r}})=0.$
		Indeed, by definition $\det(\overline{\CO_K^{\oplus r}})=\overline{\CO_K}.$ Thus, choosing~$s=1$ we see immediately that~$\hdeg(\overline{\CO_K})=0$. 
	\end{example}

\begin{example}	\label{prop:degO(1)}
    Let $\overline{\CO_{\PP^n}(1)_P}$ defined as in Example \ref{ex:ConstO(1)}, and let $P=[x_0,\ldots,x_n]\in \PP^n(K)$. Then, it follows from \cite[Proposition 9.10]{MorArGeo} that
		\[
  \widehat{\deg}(\overline{\CO_{\PP^n}(1)_P})=\sum_{\pp\subset \CO_K}\log \max_i\{|x_i|_\pp\}+\sum_{\sigma\in \Sigma_K}\log \sqrt{\sum_{i}|x_i|^2_\sigma}.
        \]
	\end{example}

 The following example can be found in \cite[Section~1.2.2]{BostTIELHVB}. 
	
	\begin{example}
		Consider the \emph{canonical module} defined as
        $$\omega_{\CO_K}:=\Hom_{\ZZ}(\CO_K,\ZZ),$$
        which is a projective $\CO_K$-module by defining $a\cdot f$ via $(a\cdot f)(b):=f(ab)$ for $a, b\in \CO_K$, $f\in \omega_{\CO_K}$. 
        The Hermitian bundle $\overline{\omega_{\CO_K}}=(\omega_{\CO_K},h)$ is defined by imposing $\|\operatorname{tr}_{K/\QQ}\|_\sigma=1$ for all~$\sigma\in \Sigma_K$, where $\operatorname{tr}_{K/\QQ}:K\to \QQ$ is the usual trace map. We call this Hermitian bundle over $S$ the \emph{canonical Hermitian bundle}. It has Arakelov degree $\hdeg(\overline{\omega_{\CO_K}})=\log|\Delta_K|$. 
	\end{example}

We refer the reader to \cite[Section~1.3.1]{BostTIELHVB} for the following properties of the Arakelov degree.
 
	\begin{prop}
		Let $\overline{E}, \overline{F}$ be Hermitian vector bundles over $S$. Then:
		\begin{enumerate}
			\item $\hdeg(\overline{E}\otimes\overline{F})=\rk F\cdot \hdeg(\overline{E})+\rk E\cdot \hdeg (\overline{F})$.
			\item $\hdeg(\overline{E}\oplus\overline{F})=\hdeg(\overline{E})+\hdeg (\overline{F})$.
			\item $\hdeg(\overline{E}^\vee)=-\hdeg(\overline{E})$.
		\end{enumerate}
	\end{prop}
	
 \subsection{Arakelov divisors over $S$}

In this section we recall the language of Arakelov divisors and their relationship with Hermitian line bundles. 
 
	\begin{defi}
		An \emph{Arakelov divisor} over~$S$ is a formal finite sum 
    \begin{equation}\label{eq: Arakelov divisor}
    D=\sum_{\pp\subset \CO_K} x_\pp \pp+\sum_{\sigma\in \Sigma_K} x_\sigma \sigma,    
    \end{equation}
    with 
    $x_\pp\in\ZZ$, and $x_\sigma\in \RR$ satisfying~$x_{\overline{\sigma}}=x_{\sigma}$ for all~$\sigma \in \Sigma_K$. 
\end{defi}

Following \cite[Chapter~I, \S 5]{NeuANT} we define
$$K_{\RR}^+:=\bigg\{(x_\sigma)\in \prod_{\sigma\in \Sigma_K}\RR: x_{\overline{\sigma}}=x_\sigma\bigg\}.$$
We then have an isomorphism of groups
$$\Div(K) \simeq \bigg(\bigoplus_{\pp \subset \CO_K} \ZZ\bigg)\times K_\RR^+,\sum_{\pp\subset \CO_K} x_\pp \pp+\sum_{\sigma\in \Sigma_K} x_\sigma \sigma\mapsto \bigg((x_\pp)_{\pp \subset \CO_K},(x_{\sigma})_{\sigma \in \Sigma_K}\bigg).$$
On $K_\RR^+$ we consider the \emph{canonical inner product}
$$\langle (x_\sigma), (y_\sigma) \rangle_{K_\RR^+}:=\sum_{\sigma \in \Sigma_K}n_\sigma x_\sigma y_\sigma,$$
where~$n_\sigma=1$ or $2$ depending on whether $\sigma$ is real or complex. This induces a canonical measure on~$K_\RR^+$ giving volume 1 to any cube generated by an orthonormal basis.
    
    We endow $\Div(K)$ with the  product topology of the discrete topology on~$\ZZ$ and the Euclidean topology on~$K_\RR^+$, and with the product measure  of the counting measure  on~$\ZZ$ and the canonical measure on~$K_\RR^+$.

\begin{remark}
    In~\cite{GeerAndSchoof2000} Arakelov divisors are defined as formal sums as in~\eqref{eq: Arakelov divisor}, but with~$\sigma$ running over the Archimedean places of~$K$. If we denote by~$v_{\sigma}=v_{\overline{\sigma}}$ the Archimedean place associated to a pair of conjugated complex embeddings~$\sigma,\overline{\sigma}\in \Sigma_K$, then the map~$x_{\sigma}\sigma+x_{\overline{\sigma}}\overline{\sigma}\mapsto 2x_{v_\sigma}$ induces an equivalence between the two notions of Arakelov divisor, which is compatible with the constructions presented in this section. In particular, the canonical measure on~$K_\RR^+\simeq \RR^{r_1+r_2}$ corresponds to the usual Lebesgue measure on~$\RR^{r_1+r_2}$.
\end{remark}

    \begin{defi}
    The \emph{degree} of the Arakelov divisor $D$ is the real number
		$$\deg(D):=\sum_{\pp\subset \CO_K}\log(\N(\pp))x_\pp+\sum_{\sigma\in \Sigma_K} x_\sigma,$$
		and we define the \emph{norm} of $D$ as $\N(D):=e^{\deg(D)}.$
	\end{defi}

\begin{notation}
    		Given $f\in K^\times$, its associated \emph{principal Arakelov divisor} is defined as $$\operatorname{div}(f):=\sum_{\pp\subset \CO_K} x_\pp \pp+\sum_{\sigma\in \Sigma_K}x_\sigma \sigma,$$ with $x_\pp=\ord_\pp(f)$ and $x_\sigma=-\log|\sigma(f)|$. 
      The quotient group of $\Div(K)$ by its subgroup of principal Arakelov divisors is denoted by $\Pic(K)$ and is called the \emph{Picard--Arakelov group}.
\end{notation}

 To each Arakelov divisor $D=\sum_\pp x_\pp \pp+\sum_\sigma x_\sigma \sigma$, we can associated a fractional ideal of $K$ by means of $D\mapsto I_D=\prod_\pp \pp^{-x_\pp}$. Then, we have a surjective homomorphism 
	$$\Div(K)\to J(K),$$
	where $J(K)$ is the group of fractional ideals of $K$. In particular, if the Archimedean part of $D$ is zero, then $\N(D)=\N(I_D)^{-1}.$ Moreover, we have the following result (see e.g. \cite[Chapter~III, Proposition 1.11 and Theorem 1.12]{NeuANT}).
	
	\begin{lem} If we denote by $\Pic^0(K)$ the subgroup of degree zero Arakelov divisor classes in $\Pic(K)$ and by $\operatorname{Cl}(K)$ the ideal class group of the number field $K$, then we have an exact sequence
		$$0\rightarrow H/\Gamma\rightarrow \Pic^0(K)\rightarrow \textup{Cl}(K)\rightarrow 0,$$
		where $H:=\{(x_\sigma)\in K_\RR^+:\sum_\sigma x_\sigma =0\}$ and $\Gamma:=\mathrm{Log}(\CO_K ^\times)$ where $\mathrm{Log}(a)=(\log|\sigma(a)|)$ for~$a\in K^\times$. In particular, $\Pic^0(K)$ is compact.
	\end{lem}

We endow~$\Pic(X)$ with the quotient measure of the product measure on~$\Div(X)$. On~$\Div^0(X)$ we consider the unique measure satisfying
\begin{equation}\label{eq: measure on Div^0}
\int_{\Div(X)}f(D)\,\textup{d}D=\int_\RR \int_{\Div^0(X)}f\left(D_0+xU\right)\textup{d}D_0\,\textup{d}x    
\end{equation}
for all~$f\in L^1(\Div(X))$, where~$U$ is the divisor\footnote{$U$ corresponds to a vector in~$K_{\RR}^+$ that is orthogonal to~$H$ and has norm 1 with respect to the canonical inner product.}
$$U:=\frac{1}{\sqrt{r_1+r_2}}\sum_{\sigma \in \Sigma_K}\frac{1}{n_\sigma}\sigma,$$
and endow~$\Pic^0 (K)$ with the corresponding quotient measure. The above lemma implies that the volume of~$\Pic^0 (K)$ 
equals
\begin{equation}\label{eq: vol(Pic^0(K))}
 \vol(\Pic^0 (K))=h_K\vol(H/\Gamma)=h_KR_K\sqrt{r_1+r_2},   
\end{equation}
where~$R_K$ and~$h_K$ are the regulator and the class number of $K$, respectively (see~\cite[Chapter~III, Proposition~7.5]{NeuANT}).

In Sections~\ref{sec: height zeta function projective} and~\ref{section:rational points} we will make use of the following formula.
 
	\begin{lem}\label{lem: integral over Pic(K)}
		Given a positive function $f\in L^1(\RR)$, 
  we have
  $$\int_{\Pic(K)} f(\deg(D))\;\textup{d}D=h_KR_K\int_\RR f(x)\,\textup{d}x.$$
	\end{lem}

	\begin{proof}
  Property~\eqref{eq: measure on Div^0} implies
  $$\int_{\Pic(K)} f(\deg(D))\,\textup{d}D=\int_\RR \int_{\Pic^0(K)}f(x\deg(U))\,\textup{d} D\,\textup{d}x=\frac{\textup{vol}(\Pic^0(K))}{\deg(U)}\int_\RR f(x)\,\textup{d}x.$$
  Then, the result follows from~\eqref{eq: vol(Pic^0(K))} together with~$\deg(U)=\sqrt{r_1+r_2}$. This proves the lemma.
	\end{proof}

	It is worth mentioning that given a Hermitian line bundle $\overline{L}=(L,h)$ over $S$, we can associate to it an Arakelov divisor in the following way: Let $s\in L$ be a non-trivial rational section and define 
	$$\operatorname{div}(s):=\sum_{\pp\subset \CO_K} v_\pp(s)\pp +\sum_{\sigma \in \Sigma_K} (-\log|s|_\sigma)\sigma.$$
	The class $D_{\overline{L}}$ of~$\operatorname{div}(s)$ in $\Pic(K)$ is independent of the choice of the section $s$. Moreover, the degree $\widehat{\deg}(\overline{L})$ of the Hermitian line bundle $\overline{L}$ over $S$ is equal to the degree $\deg(D_{\overline{L}})$ of the Arakelov divisor class $D_{\overline{L}}$.

Conversely, following \cite[p.~32]{BostTIELHVB}, given an Arakelov divisor $D=\sum_{\pp} x_\pp \pp+\sum_{\sigma} x_\sigma \sigma$, we can construct a Hermitian line bundle $\overline{\CO(D)}=(\CO_K(D),h=\{h_\sigma\})$ by defining $\CO_K(D):=I_D=\prod_{\pp} \pp ^{-x_\pp}$, and for each embedding $\sigma \in \Sigma_{K}$ imposing that $\|1\|_\sigma=
 e^{-x_\sigma}.$

 \begin{notation}
     For an Arakelov divisor~$D$ and a rational section~$f\in I_D$ we write
     $$\|f\|_D=\|f\|_{\overline{\CO(D)}}=\sqrt{\sum_{\sigma \in \Sigma_K}\|f\|_\sigma^2}=\sqrt{\sum_{\sigma \in \Sigma_K}|\sigma(f)|^2e^{-2x_\sigma}}.$$
 \end{notation}

	Note that given Hermitian line bundles $\overline{L}_1=(L_1,h_1), \overline{L_2}=(L_2,h_2)$ with trivializations $i_\pp:L_{1,\pp}\to \CO_{K,\pp}$ and $j_\pp:L_{2,\pp}\to \CO_{K,\pp}$, the corresponding trivializations for $\overline{L_1}\otimes \overline{L_2}$, $k_\pp:(L_1\otimes_{\CO_K} L_2)_\pp\to \CO_{K,\pp}$ are given by
	$k_\pp(s\otimes t)=i_\pp(s)j_\pp(t)$. Thus, $v_\pp(s\otimes t)=v_\pp(s)+v_\pp(t)$. This shows that the map~$\widehat{\Pic}(S)\to \Pic(K)$ given by~$\overline{L}\mapsto D_{\overline{L}}$ is a group homomorphism. 
 
 From the above discussion, we conclude the following.
 
	\begin{prop}\label{prop:isomorphismBetweenPicardGroups}
		
		The map $\widehat{\Pic}(S)\to \Pic(K),\;\overline{L}\mapsto D_{\overline{L}}$ is a group isomorphism.
  
	\end{prop}

By abuse of notation, we will employ this isomorphism to treat Arakelov divisor classes as Hermitian line bundles (and vice versa) when the context does not lead to confusion. For example, for an Arakelov divisor class $D\in \Pic(K)$ and $\overline{L}\in \widehat{\Pic}(S),$ we write $D\otimes \overline{L}$ to refer to the element $\overline{\CO(D)}\otimes \overline{L}\in \widehat{\Pic}(S).$
 
	\subsection{The Poisson--Riemann--Roch formula} 
	
	\begin{defi}
		Let $\overline{E}=(E, h)$ be a Hermitian vector bundle over $\Spec \ZZ$, and define 
		$$h^0(\overline{E}):=\log \sum_{v\in E} e^{-\pi\|v\|_{\overline{E}}^2},$$
  where~$\|\cdot\|_{\overline{E}}$ denotes the norm on~$E\otimes_\ZZ \CC$ associated to~$h$.   More generally, for a Hermitian vector bundle $\overline{E}=(E, h)$ over $S$ 
  we put
		$$h^0(\overline{E}):=h^0(\eta_*\overline{E}).$$
  We also 
  define the \emph{number of non-trivial sections} of $\overline{E}$ by $$\varphi(\overline{E}):=e^{h^0(\overline{E})}-1.$$
	\end{defi}

 It is worth mentioning that the previous definition coincides with the one given in \cite[Section~3]{GeerAndSchoof2000} for Arakelov divisors. More precisely, the authors consider an Arakelov divisor $D=\sum_{\pp\subset \CO_K} x_\pp \pp+\sum_\sigma x_\sigma \sigma,$ and define~$h^0(D)=\log k^0(D)$ where
$$k^0(D):=\sum_{f\in I_D} e^{-\pi \|f\|_D^2}.$$ 
Then, a simple computation shows that~$h^0(\overline{\CO(D)})=h^0(D)$. 
	\begin{lem}\label{lem:SectionsOfDirectSum}
  For $\overline{E}$ and $\overline{F}$ Hermitian vector bundles over $S$, we have that
	\begin{enumerate}
		\item $h^0(\overline{E}\oplus \overline{F})=h^0(\overline{E})+h^0(\overline{F})$, and
		\item $\varphi(\overline{E}\oplus \overline{F})=\varphi(\overline{E})+\varphi(\overline{F})+\varphi(\overline{E})\varphi(\overline{F})$.
	\end{enumerate}
	\end{lem}
	
	\begin{proof}
		Since~$\eta_*(\overline{E}\oplus \overline{F})=\eta_*(\overline{E})\oplus \eta_*(\overline{F})$, it is enough to prove item~$(1)$ for Hermitian vector bundles over~$\Spec(\ZZ)$. In that case, we have
		\begin{equation*}
			\begin{split}
    \sum_{(x,y)\in E\oplus F} 
    e^{- \pi\|(x,y)\|_{\overline{E}\oplus \overline{F}}^2} &=\sum_{(x,y)\in E\oplus F} 
    e^{- \pi\left( \|x\|_{\overline{E}}^2+\|y\|_{\overline{F}}^2\right) }\\
				&=\left( \sum_{x\in E}
    e^{- \pi\|x\|_{\overline{E}}^2}\right) \left( \sum_{y\in F}
    e^{- \pi\|y\|_{\overline{F}}^2}\right). 
			\end{split}
		\end{equation*}
  Taking logarithms we conclude~$(1)$. Item~$(2)$ is a direct consequence of~$(1)$. This proves the lemma. 
	\end{proof}
 
	The following formula follows from Lemma~\ref{lem:SectionsOfDirectSum}(2) by induction.
	\begin{cor}\label{cor:varphiOfOplus}
		Let $\sigma_1, \ldots, \sigma_n$ be the elementary symmetric polynomials in $n$ variables and let $\overline{E}_1, \ldots, \overline{E}_n$ be Hermitian vector bundles over $S$. Then
		\[\varphi(\overline{E}_1\oplus\ldots\oplus \overline{E}_n)=\sum_{i=1}^n \sigma_i(\varphi(\overline{E}_1), \ldots, \varphi(\overline{E}_n)).\]
		\end{cor}

We can now state the Poisson--Riemann--Roch formula for Hermitian vector bundles over the arithmetic curve $S$. See 
\cite[Section~2.2.2]{BostTIELHVB} for details.
 
	\begin{thm}\label{thm:RiemannRoch}
 Let~$\overline{E}$ be a Hermitian vector bundle over $S$. Then
		$$h^0(\overline{E})-h^0(\overline{\omega_{\CO_K}}\otimes \overline{E}^\vee)=\widehat{\deg}(E)-\frac{1}{2}(\log|\Delta_K|)\cdot \operatorname{rk}(\overline{E}).$$
		Equivalently, we have $ \varphi(\overline{E})=\left( \varphi(\overline{E}^\vee\otimes \overline{\omega_{\CO_K}})+1\right) \N(\overline{E})|\Delta_K|^{-\frac{\operatorname{rk}(\overline{E})}{2}}-1.$
			\end{thm}

In \cite[Section 5, Corollary 1]{GeerAndSchoof2000} the authors prove the following bound for the number of non-trivial sections of Arakelov divisors with bounded degree.
 
	\begin{prop}\label{prop:cota_Geer_Schoof}
 Let $C\in \RR$ and let $D$ be an Arakelov divisor over $S$ with $\deg(D)\leq C$. Then
		$$\varphi(D):=\varphi(\overline{\mathcal{O}(D)})\leq \beta e^{-\pi n_Ke^{-\frac{2}{n_K}\deg(D)}},$$ 
		for some $\beta>0$ depending on $C$ and $K$, where $n_K=[K:\mathbb{Q}]$. 
	\end{prop}
	\begin{remark} In the proof of \cite[Section~5, Corollary 1]{GeerAndSchoof2000}, the authors assumed that $\deg(D)\leq \frac{1}{2}\log(|\Delta_K|)$. Their proof can be adapted 
 to Arakelov divisors with $\deg(D)\leq C$ by considering
 $$u=\frac{1}{n_K}(C-\deg(D)) \text{ and }D'=D+\sum_\sigma u \sigma,$$
 instead of the~$u$ and~$D'$ used in their proof of \cite[Proposition 2]{GeerAndSchoof2000}. 

	\end{remark}

  In order to deal with Hirzebruch--Kleinschmidt varieties, we will need the following bound.  
 
	\begin{prop}\label{prop:boundForNumberSections}
		Let $\overline{E}$ be a split Hermitian vector bundle over~$S$, i.e., $\overline{E}=\overline{L}_1\oplus\cdots\oplus \overline{L}_r$ where each $\overline{L}_i$ is a Hermitian line bundle, and let $\overline{L}\in \widehat{\Pic}(S)$ 
  such that~$\widehat{\deg}(\overline{L})\leq C$ for some~$C\in \RR$. Then, there exist $\beta ,\gamma>0$ depending on $C,K$ and~$\overline{E}$, such that 
		\[
        \varphi(\overline{E}\otimes \overline{L})\leq \beta e^{-\gamma e^{-\frac{2}{n_K}{\widehat{\deg}}(\overline{L})}}
        \]
        where $n_K=[K:\mathbb{Q}]$.
		\end{prop}
	
	\begin{proof}
  We have
  $$\varphi(\overline{E}\otimes \overline{L})= \varphi ((\overline{L}_1\oplus\cdots \oplus \overline{L}_r)\otimes \overline{L}).$$ As $\varphi ((\overline{L}_1\oplus\cdots \oplus \overline{L}_r)\otimes \overline{L})$ depends polynomially on $\varphi(\overline{L}_i\otimes \overline{L})$ by Corollary \ref{cor:varphiOfOplus}, it is enough to prove the bound in the particular case $\overline{E}=\overline{L}_i$. 
  Since $\widehat{\deg}(\overline{L}_i\otimes \overline{L}) =\widehat{\deg}(\overline{L}_i)+\widehat{\deg}(\overline{L})\leq C+\widehat{\deg}(\overline{L}_i)$, it follows from Proposition \ref{prop:cota_Geer_Schoof} that there exists $\beta_i>0$, depending on~$C$, $K$ and~$\overline{L}_i$, such that 
		\begin{equation*}
				\varphi(\overline{L}_i\otimes \overline{L})\leq \beta_i e^{-\pi n_Ke^{-\frac{2}{n_K}\widehat{\deg}(\overline{L_i}\otimes \overline{L})}}
    =  \beta e^{-\gamma_i e^{-\frac{2}{n_K}\widehat{\deg}(\overline{L})}}
		\end{equation*}
		were $\gamma_i:=\pi n_K e^{-\frac{2}{n_K}\widehat{\deg}(\overline{L_i})}$. 
  This proves the desired result. 
	\end{proof}

In the particular case when 
$\overline{L}$ is a Hermitian line bundle, the following uniform bound can be obtained (see \cite[Proposition 2.7.3]{BostTIELHVB}).

	\begin{prop}\label{prop:boundForHLB}
		Let  $\theta\in \RR_{\geq0}$ and $\overline{L}$ be a Hermitian line bundle over $S$ such that $\widehat{\deg}(\overline{L})\leq \theta$. Then~$h^0(\overline{L})\leq 1+\theta$. In particular, $\varphi(\overline{L})\leq e^{1+\theta}$.
	\end{prop}

\section{Height zeta function of the projective space}\label{sec: height zeta function projective}

In this section, we introduce a zeta function of the field $K$ defined in~\cite{GeerAndSchoof2000}. This zeta function will allow us to  study the analytic properties of the height zeta function of the projective space via a suitable integral representation. 

We first define the \emph{effectivity} $e(D)$ of an Arakelov divisor $D=\sum_\pp x_\pp\pp+\sum_\sigma x_\sigma \sigma$ in $\operatorname{Div}(K)$ as
 $$e(D):=\left\{\begin{array}{ll}
  e^{-\pi \|1\|^2_D}=e^{-\pi\sum_{\sigma}e^{-2x_\sigma}}    & \text{if $x_\pp\geq 0$ for all~$\pp\subset \CO_K$}, \\
   0   & \text{otherwise}.
 \end{array}
  \right.$$

	In \cite[Section~4]{GeerAndSchoof2000}, the authors define the zeta function associated to $K$ as
	$$\xi_K(s):=\int_{\Div(K)} \N(D)^{-s}e(D)\;\textup{d}D, \quad s\in \CC, \Re(s)>1,$$
	and they prove that\footnote{There is a misprint in the first power of $2$ appearing in the third line of the computation leading to the formula for $\xi_K(s)$ in \cite[p.~388]{GeerAndSchoof2000}. In the computation of the integral over~$t_\sigma$ for~$\sigma$ real, the factor 2 should appear in the denominator.} $\xi_K(s)=2^{-r_1}\left( \pi^{-s/2}\Gamma(s/2)\right)^{r_1} \left( (2\pi)^{-s}\Gamma(s)\right)^{r_2}\zeta_K(s)$, 	where
 $$\zeta_K(s):=\sum_{\left\{0\right\}\neq J\subseteq \CO_K}N(J)^{-s}$$ 
 is the Dedekind zeta function of the field $K$. In particular, $\xi_K(s)$ has meromorphic continuation to~$s\in \CC$.
 Moreover, the authors show that 
    \begin{equation}
		\xi_K(s)=\frac{1}{w_K} \int_{\Pic(K)}\N(D)^{-s}\varphi(D)\;\textup{d}D.
		\label{eq:equationOverPic}
	\end{equation}


 Let~$n\geq 1$ be an integer. Given $P=[x_0,\ldots,x_n]\in\PP ^n(K)$, we define the \emph{standard height} of $P$ as $$H_{\PP^n}(P):=\prod_{\pp\subset \CO_K}\max_i\{|x_i|_\pp\}\cdot\prod_{\sigma\in \Sigma_K} \sqrt{\sum_i|x_i|^2_\sigma}.$$
 It follows from Example $\ref{prop:degO(1)}$ that
    \[
    H_{\PP^n}(P)=\N(\overline{\CO_{\PP^n}(1)_P}) \textup{ for every }P\in \PP^n(K). 
    \]
	 The associated \emph{height zeta function} is 
	$$\Z_{\PP^n}(s):=\sum_{P\in \PP^n(K)}H_{\PP^n}(P)^{-s}=\sum_{P\in \PP^n(K)} \N(\overline{\CO_{\PP^n}(1)_P})^{-s},$$
 defined for~$s\in \CC$ with~$\Re(s)>n+1$ (the series converges absolutely and uniformly on compact subsets of this domain). 
We will study this function by means of Arakelov geometry.

As in~\cite[Section~3.2]{Maruyama2015}, we will work with a $K$-vector space $V$ of dimension $n+1$ that contains a  complete $\CO_K$-lattice $E$ which is the underlying finitely generated projective $\mathcal{O}_K$-module of a Hermitian vector bundle $\overline{E}=(E,h)$. Analogous to the construction carried out in Example \ref{ex:ConstO(1)}, given a point $P\in \PP(V)$, we define its height by
 	$$H_{\mathbb{P}(V)}(P):=\N(\overline{\CO_{\mathbb{P}(V)}(1)_P}).$$
  We also denote by~$\Z_{\PP(V)}(s)$ the corresponding height zeta function. In particular, considering $V=K^{\oplus (n+1)}$ we have $H_{{\mathbb{P}(K^{\oplus (n+1)})}}(P)=H_{\PP^n}(P)$.

	Recall that if $D \in \Pic(K)$, we denote the Hermitian line bundle $\overline{\mathcal{O}(D)}$ simply by $D$. As explained before, the key idea is to express the height zeta function~$\Z_{\PP^n}(s)$ as a suitable integral. To do so, we note that \eqref{eq:equationOverPic} implies that
	\begin{equation*}
		\begin{split}
			w_K\xi_K(s)H_{\mathbb{P}(V)}(P)^{-s}&=\int_{\Pic(K)}\N(D)^{-s}\varphi(D)H_{\mathbb{P}(V)}(P)^{-s}\,\textup{d}D\\
			&=\int_{\Pic(K)}(\N(D)N(\overline{\CO_{\mathbb{P}(V)}(1)_P}))^{-s}\varphi(D)\,\textup{d}D\\
			&=\int_{\Pic(K)}\N(D\otimes\overline{\CO_{\mathbb{P}(V)}(1)_P})^{-s}\varphi(D)\,\textup{d}D\\
			&=\int_{\Pic(K)}\N(D)^{-s}\varphi(D\otimes\overline{\CO_{\mathbb{P}(V)}(-1)_P})\,\textup{d}D.
		\end{split}
	\end{equation*}

	If we fix $D\in \Pic(K)$ and we let $P$ run through $\PP(V)$, then $D\otimes \overline{\CO_{\mathbb{P}(V)}(-1)_P}$ runs through all subline bundles of $\overline{\CO(D)}\otimes \overline{E}$. Therefore, the above formula implies that
	\begin{equation}
		\begin{split}
			w_K\xi_K(s)\Z_{\PP(V)}(s)&=\int_{\Pic(K)}\N(D)^{-s}\sum_{P\in \PP(V)}\varphi(D\otimes \overline{\CO_{\mathbb{P}(V)}(-1)_P})\,\textup{d}D\\
			&=\int_{\Pic(K)}\N(D)^{-s}\varphi(D\otimes \overline{E})\,\textup{d}D.
		\end{split}
		\label{eq:zetaFunctionMultipliedbyFactors}
	\end{equation}

\begin{notation}
    We denote by
\[
\Pic(K)_{-}:=\left\{D\in \Pic(K):\N(D)\leq \sqrt{|\Delta_K|}\right\}
\]
the set of Arakelov divisor classes with norm bounded above by $\sqrt{|\Delta_K|}$.
    
\end{notation}

	\begin{prop}\label{prop:LemmaOfIntegraloverPic}
		Let $V$ be a $K$-vector space of dimension $n+1$ containing a complete $\CO_K$-lattice $E$ which is the underlying finitely generated projective $\mathcal{O}_K$-module of a Hermitian vector bundle $\overline{E}=(E,h)$. Then:
\begin{enumerate}
    \item The integral
    \[
    \int_{\Pic(K)_{-}}\N(D)^{-s}\varphi(D\otimes \overline{E})\,\textup{d}D
    \]
    converges absolutely and uniformly for~$s$ in compact subsets of~$\CC$.
    \item For~$\Re(s)>n+1$ we have
    \begin{equation*}
			\begin{split}
	w_K\xi_K(s)\Z_{\PP(V)}(s)			
				&=\int_{\Pic(K)_{-}}\N(D)^{-s}\varphi(D\otimes \overline{E})\,\textup{d}D\\
				&\hspace{1cm}+\N(\overline{E})|\Delta_K|^{\frac{(n+1)}{2}-s}\int_{\Pic(K)_{-}}\N(D)^{s-(n+1)}\varphi(D\otimes \overline{E}^\vee)\,\textup{d}D\\
				&\hspace{2cm}+ R_Kh_K|\Delta_K|^{-\frac{s}{2}}\left(\frac{\N(\overline{E})}{s-(n+1)}-\frac{1}{s}\right).
			\end{split}
		\end{equation*}
  \end{enumerate}	
		\end{prop}
	\begin{proof}
Item~$(1)$ follows from Proposition~\ref{prop:boundForNumberSections}.  
Indeed, since Arakelov divisor classes $D\in \Pic(K)_{-}$ satisfy $\widehat{\deg}(D)\leq\frac{1 }{2}\log(|\Delta_K|)$, there are constants $\beta, \gamma>0$ depending on $K$ and $\overline{E}$ such that 
		\begin{equation*}
				\int_{\Pic(K)_{-}}\left| \N(D)^{-s}\right| \varphi(D\otimes \overline{E})\,\textup{d}D\leq \int_{\Pic(K)_{-}}\left| \N(D)^{-s}\right| \beta e^{-\gamma e^{-\frac{2}{n_K}\widehat{\deg}(D)}}\,\textup{d}D.
		\end{equation*}
  Using Lemma~\ref{lem: integral over Pic(K)} we get
  \begin{equation*}
    \int_{\Pic(K)_{-}}\left| \N(D)^{-s}\right| \varphi(D\otimes \overline{E})\,\textup{d}D\leq  R_K h_K \beta \int_{-\infty}^{\frac{1}{2}\log(|\Delta_K|)} e^{-\Re(s)x-\gamma e^{-\frac{2}{n_K}x}} \textup{d}x,
  \end{equation*}
  and this last integral converges uniformly for~$s$ in compact subsets of~$\CC$. This proves item~$(1)$.
  
  Now, if we define $\Pic(K)_+:=\{D\in \Pic(K):\N(D)\geq \sqrt{|\Delta_K|}\}$, then by~\eqref{eq:zetaFunctionMultipliedbyFactors} we have
		\begin{equation*}
			\begin{split}
				w_K\xi_K(s)\Z_{\PP(V)}(s)	
				=&\int_{\Pic(K)_{-}}\N(D)^{-s}\varphi(D\otimes \overline{E})\,\textup{d}D+\int_{\Pic(K)_+}\N(D)^{-s}\varphi(D\otimes \overline{E})\,\textup{d}D.
			\end{split}
		\end{equation*}
In order to compute the integral over~$\Pic(K)_+$, we consider the change of variables $D\mapsto \omega_{\CO_K}\otimes D^\vee$ to get
		\begin{equation*}
			\begin{split}
				\int_{\Pic(K)_+}\N(D)^{-s}\varphi(D\otimes \overline{E})\,\textup{d}D&=\int_{\Pic(K)_{-}}\N(\omega_{\CO_K}\otimes D^\vee)^{-s}\varphi(\omega_{\CO_K}\otimes D^\vee\otimes \overline{E})\,\textup{d}D\\
				&=\int_{\Pic(K)_{-}}|\Delta_K|^{-s}\N(D)^{s}\varphi(\omega_{\CO_K}\otimes D^\vee\otimes \overline{E})\,\textup{d}D.
			\end{split}
			\label{eq:convPic(K)+}
		\end{equation*}
		By Theorem \ref{thm:RiemannRoch} we have
		\begin{equation*}
			\begin{split}
				\varphi(D^\vee \otimes\omega_{\CO_K}\otimes \overline{E})
    &=e^{h^0(\omega_{\CO_K}\otimes (D\otimes \overline{E}^\vee)^\vee)}-1\\
				&=e^{h^0(D\otimes \overline{E}^\vee)-\deg(D\otimes \overline{E}^\vee)+\log|\Delta_K|\cdot \frac{\operatorname{rk}(D\otimes \overline{E}^\vee)}{2}}-1\\
				&=\left( \varphi(D\otimes \overline{E}^\vee)+1\right)\N(D)^{-(n+1)}\N(\overline{E})|\Delta_K|^{\frac{n+1}{2}}-1.
			\end{split}
		\end{equation*}
		Therefore,
		\begin{equation*}
			\begin{split}
				\int_{\Pic(K)_+}\N(D)^{-s}	&\varphi(D\otimes \overline{E})\,\textup{d}D\\
				&=\int_{\Pic(K)_{-}}\N(D)^{s}|\Delta_K|^{-s}\left(\left( \varphi(D\otimes \overline{E}^\vee)+1\right)\N(D)^{-(n+1)}\N(\overline{E})|\Delta_K|^{\frac{n+1}{2}}-1 \right) \,\textup{d}D\\
				&=\N(\overline{E})|\Delta_K|^{\frac{n+1}{2}-s}\int_{\Pic(K)_{-}}\N(D)^{s-(n+1)}\varphi(D\otimes \overline{E}^\vee)\,\textup{d}D\\
				&\hspace{2cm}+\N(\overline{E})|\Delta_K|^{\frac{n+1}{2}-s}\int_{\Pic(K)_{-}}\N(D)^{s-(n+1)}\,\textup{d}D\\
				&\hspace{3cm}-|\Delta_K|^{-s}\int_{\Pic(K)_{-}}\N(D)^{s}\,\textup{d}D.
			\end{split}
		\end{equation*}
		Now, for $\Re(s)>n+1$ we have (using Lemma~\ref{lem: integral over Pic(K)})
		\begin{equation*}
			\begin{split}
				\int_{\Pic(K)_{-}}\N(D)^{s-(n+1)}\,\textup{d}D&=\int_{\Pic(K)_{-}}e^{(s-(n+1))\deg(
					D)}\,\textup{d}D \\
     &=R_Kh_K\int_{-\infty}^{\frac{1}{2}\log(|\Delta_K|)}e^{(s-(n+1))x} \textup{d}x\\
				&=R_Kh_K\frac{|\Delta_K|^{\frac{s-(n+1)}{2}}}{s-(n+1)}.
			\end{split}
		\end{equation*}
		Analogously, for $\Re(s)>0$, we have
		$$\int_{\Pic(K)_{-}}\N(D)^{s}\,\textup{d}D=R_Kh_K
  \frac{|\Delta_K|^{\frac{s}{2}}}{s}.$$
  This proves item~$(2)$ and completes the proof of the proposition.
	\end{proof}

\begin{remark}
Proposition~\ref{prop:LemmaOfIntegraloverPic} also holds in the case~$n=0$, $V=K$ and~$E=\overline{\CO_K}$, in which case~$\Z_{\PP(V)}=1$. In particular:
    \begin{equation}\label{eq extension of xi_K}
    \begin{split}
	w_K\xi_K(s)	&=\int_{\Pic(K)_{-}}\N(D)^{-s}\varphi(D)\,\textup{d}D +|\Delta_K|^{\frac{1-s}{2}}\int_{\Pic(K)_{-}}\N(D)^{s-1}\varphi(D)\,\textup{d}D\\
				&\hspace{2cm}+ R_Kh_K|\Delta_K|^{-\frac{s}{2}}\left(\frac{1}{s-1}-\frac{1}{s}\right),
			\end{split}\end{equation}
   and this gives the meromorphic continuation of~$\xi_K(s)$ to~$\CC$ (as in~\cite[Section~4]{GeerAndSchoof2000}).
\end{remark}

	As consequence of Proposition \ref{prop:LemmaOfIntegraloverPic} we have the following result (see \cite[Theorem~3.2]{Maruyama2015}).
	\begin{thm}[Maruyama]\label{thm Maruyama}
		Let $V$ be a $K$-vector space of dimension $n+1$ containing a  complete $\mathcal{O}_K$-lattice $E$ which is the underlying finitely generated projective $\mathcal{O}_K$-module of a Hermitian vector bundle $\overline{E}=(E,h)$. Then, the function $\Z_{\PP(V)}(s)$
 has  meromorphic continuation to the whole complex plane, which is holomorphic for $\Re(s)>1,s\neq n+1$, and with a simple pole at $s=n+1$. Moreover, we have
		$$\operatorname{Res}_{s=n+1}\Z_{\PP(V)}(s)=\frac{R_Kh_K\N(\overline{E})}{w_K|\Delta_K|^{\frac{n+1}{2}}\xi_K(n+1)}.$$
	\end{thm}

Choosing~$V=K^{\oplus(n+1)}$ and $\overline{E}=\overline{\CO_K^{\oplus(n+1)}}$, in which case $\N(\overline{E})=1$, we obtain the following corollary as an application of Theorem~\ref{thm:tauberianThm}.
 
	\begin{cor}[Schanuel's estimate] \label{cor:SchanuelTheorem}
 Let $N(\PP^n,B):=\#\{P\in \PP^n(K):H_{\PP^n}(P)\leq B\}$. Then
		$$N(\PP^n,B)\sim CB^{n+1} \quad \text{as }B\to \infty,$$
		with
		$$C:=\frac{R_Kh_K}{(n+1)w_K|\Delta_K|^{\frac{n+1}{2}}\xi_K(n+1)}.$$
	\end{cor}


	\begin{remark}
 \label{remark:metrization}
		Note that the asymptotic constant given above is in general different from  to the one obtained by Schanuel in \cite{Sch79}. This is due to the fact that Schanuel uses an $\ell^\infty$ 
  norm on the non-Archimedean places, while we use an $\ell^2$ norm. 
  Also, compare this result with the one obtained by  Guignard in \cite[Cor. 3.4.2]{GuignarZetaFunctions}, taking into account that Guignard defines $\N(\overline{E})=e^{-\widehat{\deg}(\overline{E})}.$
	\end{remark}

	\section{Counting rational points on Hirzebruch--Kleinschmidt varieties}\label{section:rational points}

In this section, we construct an Arakelov height function~$H_L$ associated to a big line bundle class~$L\in \Pic(X_d(a_1,\ldots,a_r))$, and describe the asymptotic growth of the number~$N(U,H_L,B):=\#\{P\in U(K):H_L(P)\leq B\}$, where~$U=U_d(a_1,\ldots,a_r)$ is the good open subset of~$X_d(a_1,\ldots,a_r)$, as defined in the Introduction. The main results are Theorems~\ref{thm:generalVersionTheorem} and~\ref{thm:degenerate_case_general_theorem}, which are used in Section~\ref{subsection proof of main thm anticanonical} to prove Theorem~\ref{main thm anticanonical}. Finally, in Section~\ref{sec accumulation of rational points}, we briefly discuss subvarieties that accumulate more rational points than others and provide criteria to determine when this occurs.




\subsection{Heights induced by big line bundles}\label{subsection:general case} 

Let $X=X_d(a_1,\ldots,a_r)$ be a Hirzebruch--Kleinschmidt variety  of dimension $d=r+t-1$ defined over the number field $K$ (see Definition \ref{defi:HK varieties}). Recall that 
 $$\pi: X=\PP(\CO_{\PP^{t-1}} \oplus \CO_{\PP^{t-1}}(-a_r)\oplus \CO_{\PP^{t-1}}(a_1-a_r)\oplus \cdots\oplus \CO_{\PP^{t-1}}(a_{r-1}-a_r) )\to \PP^{t-1}$$ 
 is a projective vector bundle over $\PP^{t-1}$.  

 Let us define 
 \begin{equation}\label{eq def of W}
   \mathscr{W}:= \CO_{\PP^{t-1}} \oplus \CO_{\PP^{t-1}}(a_r)\oplus  \cdots\oplus \CO_{\PP^{t-1}}(a_r-a_{r-1}),  
 \end{equation}
	and recall that for $P\in X(K)$ the fiber 
 $$\CO_X(-1)_P=\ell\subseteq \left( \CO_{\PP^{t-1}} \oplus \CO_{\PP^{t-1}}(-a_r)\oplus \CO_{\PP^{t-1}}(a_1-a_r)\oplus \cdots\oplus \CO_{\PP^{t-1}}(a_{r-1}-a_r)  \right) ^\vee_{\pi(P)}, $$ 
 is given by the one-dimensional subspace $\ell$ of the $(r+1)$-dimensional vector space
    \[
    \left(\CO_{\PP^{t-1}} \oplus \CO_{\PP^{t-1}}(-a_r)\oplus  \cdots\oplus \CO_{\PP^{t-1}}(a_{r-1}-a_r) \right) ^\vee_{\pi(P)}=
    \mathscr{W}_{\pi(P)},
       \]
	   corresponding to the point $P$ in $\pi^{-1}(\pi(P))=\PP(\mathscr{W}_{\pi(P)})$. The vector space~$\mathscr{W}_{\pi(P)}$ contains the Hermitian vector bundle 
    $$ \overline{(\CO_{\PP^{t-1}})_{\pi(P)}} \oplus \overline{\CO_{\PP^{t-1}}(a_r)_{\pi(P)}}\oplus  \cdots\oplus \overline{\CO_{\PP^{t-1}}(a_r-a_{r-1})_{\pi (P)}},$$
    which we denote by~$\overline{\mathscr{W}_{\pi(P)}}$ for simplicity. By endowing $\ell \cap \overline{\mathscr{W}_{\pi(P)}}$ with the restriction of the  Hermitian forms on~$\overline{\mathscr{W}_{\pi (P)}}$, we obtain the Hermitian line bundle~$\overline{\CO_X(-1)_P}$. As usual, by taking duals and tensor powers we  define~$\overline{\CO_X(a)_P}$ for any~$a\in \ZZ$. 

Given $L=\lambda h+\mu f \in \Pic(X)$ big and~$P\in X(K)$, we put    
    $$\overline{L_P}:=\overline{\CO_X(\lambda)_P}\otimes \overline{\CO_ {\PP ^{t-1}}(\mu)_{\pi (P)}}.$$
This induces an \emph{adelic metric} on~$L$ as defined in~\cite[\emph{D\'efinition~1.4}]{Pey02}. We refer to this as the \emph{standard metric} on~$L$.
    
We can now define the \emph{standard height function}~$H_L$ over $X(K)$ associated to~$L$ as
    \[
    H_L(P):=\N(\overline{L_P}).
    \]
    More explicitly, we have 
\begin{equation}\label{eq H_L of P}
          H_L(P)= e^{\lambda \widehat{\deg} \left(\overline{\CO_{X}(1)_P}\right) }e^{\mu \widehat{\deg}\left(  \overline{\CO_{\PP ^{t-1}}(1)_{\pi(P)}} \right)}= H_{\PP(\mathscr{W}_{\pi(P)})}(P)^{\lambda }H_{\PP ^{t-1}}\left( \pi(P)\right)^{\mu}.
   \end{equation}
Associated to~$L=\lambda h+\mu f $ as above, we define
\begin{equation}\label{eq lambda_D and mu_D}
\lambda_L:=\frac{r+1}{\lambda}, \quad \mu_L:=\frac{(r+1)a_r+t-|\mathbf{a}|}{\mu}.
\end{equation} 
Then, it easily follows from Proposition~\ref{prop effective cone} that
\begin{equation}\label{eq a(L) and b(L)}
     a(L)=\max\{\lambda_L,\mu_L\}\quad \text{and}\quad b(L)=\left\{\begin{array}{ll}
2     &  \text{if }\lambda_L=\mu_L, \\
1     &  \text{if }\lambda_L\neq \mu_L.
\end{array}
\right.
\end{equation}
    
As in the Introduction, we restrict our attention to rational points in a specific open subset~$U\subseteq X$. This is done in order to ensure that~$N(U,H_L,B)$ is finite for all~$B>0$, and to avoid possible proper subvarieties with too many rational points.

Recall that in Section~\ref{sec restriction of divisors} we defined the projective subbundle~$F=\mathbb{P}(\mathscr{Y})\subset X_d(a_1,\ldots,a_r)$ where
$$\mathscr{Y}=\CO_{\PP^{t-1}}(-a_r)\oplus \CO_{\PP^{t-1}}(a_1-a_r)\oplus \cdots\oplus \CO_{\PP^{t-1}}(a_{r-1}-a_r).$$
The following definition was given in the Introduction.

\begin{defi}\label{defi:good open}
Given integers $r\geq 1$, $t\geq 2$ and $0\leq a_1 \leq \cdots \leq a_r$, we define the \emph{good open subset} of $X=X_d(a_1,\ldots,a_d)$ as 
$$U_d(a_1,\ldots,a_r):=X_d(a_1,\ldots,a_r)\setminus F.$$
\end{defi}

\begin{remark}
Given a big line bundle class~$L\in \Pic(X)$, it is know that there exists a dense open subset~$U_L\subseteq X$ such that $N(U_L,H_L,B)$ is finite for every~$B>0$ (see e.g.~\cite[Proposition~2.12]{Peyre2021}). Our good open subset~$U$ serves as such a dense open~$U_L$ for every big~$L$. 
\end{remark}

The first main result of this section is the following theorem, where we assume~$a_r>0$. The easier case when~$a_r=0$ is presented later in this section (see Theorem~\ref{thm:degenerate_case_general_theorem}). Recall that, for~$m\geq 1$, we defined~$\Z_{\PP^m}(s)$ as the height zeta function of the projective space~$\PP^m$ with respect to the standard height function (see Section~\ref{sec: height zeta function projective}). Here, we extend this definition by putting~$\Z_{\PP^m}(s):=1$ (resp.~$0$) if~$m=0$ (resp.~$m=-1$).

\begin{thm}
 \label{thm:generalVersionTheorem}
     Let $X=X_d(a_1,\ldots,a_r)$ be a Hirzebruch--Kleinschmidt variety over the number field~$K$ of dimension~$d=r+t-1$, and let~$L=\lambda h+\mu f\in \Pic(X)$ big. Assume~$a_r>0$. Then, we have
        $$N(U,H_L,B) \sim C_{L,K}B^{a(L)}\log(B)^{b(L)} \quad \text{as }B\to \infty,$$
     with~$C_{L,K}$ given by
     $$\left\{
     \begin{array}{ll}
    \vspace{0.2cm}  \frac{R_K^2h_K^2|\Delta_K|^{-\frac{(d+2)}{2}}}{w_K^2 (r+1) \mu \xi_K(r+1)\xi_K(t)}    &  \text{if }\lambda_L=\mu_L, \\
   \vspace{0.2cm}     \frac{R_Kh_K|\Delta_K|^{-\frac{r+1}{2}}}{w_K (r+1) \xi(r+1)} \Z_{\mathbb{P}^{t-1}}\left( \mu \lambda_L+|\mathbf{a}|-(r+1)a_r\right)   & \text{if }\lambda_L>\mu_L,\\
    \begin{array}{l}
            \frac{R_Kh_K|\Delta_K|^{-\frac{t-N_X+(r+1)}{2}}\xi_K(\lambda \mu_L+N_X-(r+1))}{w_K ((r+1)a_r+t-|\mathbf{a}|) \xi_K(\lambda \mu_L)\xi_K(t)}  \\
      \vspace{0.1cm}   \hspace{1cm} \times \big( \Z_{\PP^{N_X-1}}(\lambda \mu_L+N_X-(r+1))-\Z_{\PP^{N_X-2}}(\lambda \mu_L+N_X-(r+1)) \big)
    \end{array}  \vspace{0.2cm}    
          & \text{if }\lambda_L<\mu_L,
     \end{array}\right.$$
     where~$N_X:=\# \{i\in \{1,\ldots,r\}:a_i=a_r\}$.
 \end{thm}

In the proof of Theorem~\ref{thm:generalVersionTheorem} below, we study the analytic properties of the height zeta function
$$\Z_{U,L}(s):=\sum_{P\in U(K)}H_L(P)^{-s}$$
associated  to the big line bundle class~$L$ and the good open subset~$U:=U_d(a_1,\ldots,a_r)$, and we make use of the following three lemmas. 

\begin{lem}\label{lem: convergencia integral varphi(D)^m}
For an integer~$m\geq 0$, define
\begin{equation}\label{eq varphi_m}
\varphi_m(D):=\left\{\begin{array}{ll}
\varphi(D)     & \text{if }m=0, \\
\varphi(D)\varphi(D^{\oplus m})     & \text{if }m\geq 1.
\end{array}\right.
\end{equation}
Then, the following properties hold:
\begin{enumerate}
    \item For $D\in \Pic(K)_{-}$ we have $\varphi(D^{\oplus m})=O(1)$, with an implicit constant depending only on~$m$ and on the base field~$K$.
    \item The integral
    \[
    \int_{\Pic(K)_{-}}\N(D)^{-s}\varphi_m(D)\,\textup{d}D
    \]
    converges absolutely and uniformly for~$s$ in compact subsets of~$\CC$.
\end{enumerate}
\end{lem}
\begin{proof}
Item~$(1)$ follows from Proposition~\ref{prop:boundForNumberSections} since~$\varphi(D^{\oplus n})=\varphi(D\otimes \overline{\CO_K^{\oplus n}})$ is bounded above, for~$D\in \Pic(K)_{-}$, by a constant depending only on~$K$ and~$n$, hence we can use this fact for~$n=1$ and~$n=m$. Item~$(2)$ follows  directly from Proposition~\ref{prop:LemmaOfIntegraloverPic}(1) and the fact that~$\varphi(D^{\oplus m})$ is bounded. This proves the lemma.
\end{proof}

\begin{lem}\label{lem: bound for positive split bundles}
Given positive integers~$0<b_1\leq b_2\leq \cdots \leq b_n$ and~$Q\in \PP^{t-1}$, define
$$\overline{E_Q}:=\overline{ \CO_{\PP^{t-1}}(b_1)_Q}\oplus \cdots\oplus  \overline{(\CO_{\PP^{t-1}}(b_n))_Q},$$
and~$|\mathbf{b}|:=\sum_{i=1}^n b_i$. Then, for every integer~$m\geq 0$, every compact subset~$\mathcal{K}\subset \CC$ and every~$s\in \mathcal{K}$, we have
\begin{equation*}
    \begin{split}
& \int_{\Pic(K)_{-}}\N(D)^{-s}\varphi_m(D)\varphi\left(D\otimes \overline{E_Q}\right) \,\textup{d}D\\
&\hspace{2cm} = |\Delta_K|^{-\frac{n}{2}}H_{\PP^{t-1}}(Q)^{|\mathbf{b}|}\int_{\Pic(K)_{-}}\N(D)^{n-s}\varphi_m(D) \,\textup{d}D+O(H_{\PP^{t-1}}(Q)^{|\mathbf{b}| -b_1}),        
    \end{split}
\end{equation*}
with an implicit constant depending only on~$m,\mathcal{K},b_1,\ldots,b_n$ and on the base field~$K$, where~$\varphi_m(D)$ is defined in~\eqref{eq varphi_m}.
\end{lem}

\begin{proof}
Let us define
\begin{equation*}
\begin{split}
    \mathrm{I}_Q&:=\{D\in \Pic(K)_{-} : N(D)H_{\PP^{t-1}}(Q)^{b_1}\leq \sqrt{|\Delta_K|}\},\\
    \mathrm{II}_Q&:=\{D\in \Pic(K)_{-} : \sqrt{|\Delta_K|}\leq N(D)H_{\PP^{t-1}}(Q)^{b_1}\},
    \end{split}
\end{equation*}
and for~$i\in \{1,\ldots, n\}$ put~$\overline{E_Q^{(i)}}:=\overline{\CO_{\PP^{t-1}}(b_i)_Q}$. Fix a compact subset~$\mathcal{K}\subset \CC$ and assume~$s\in \mathcal{K}$. 
In what follows, all terms of the form~$O(\ldots)$ are meant to have implicit constants depending only on~$m,\mathcal{K},b_1,\ldots,b_n$ and the base field~$K$.

First, from Lemma~\ref{lem: convergencia integral varphi(D)^m}(1) with~$m=1$ it follows that there exists a constant~$C_1\geq 1$, depending only on the base field~$K$, such that~$\varphi(D)\leq C_1$ for all~$D\in \Pic(K)_{-}$. Equivalently, $\varphi\left(D\otimes \overline{E_Q^{(1)}}\right)\leq C_1$ for all~$D\in  \mathrm{I}_Q$. Also, by Proposition~\ref{prop:boundForHLB} we have
$$\varphi\left(D\otimes \overline{E_Q^{(i)}}\right)\leq e^{1+\frac{1}{2}\log(|\Delta_K|)+b_i\log(H_{\PP^{t-1}}(Q))}=e|\Delta_K|^{\frac{1}{2}}H_{\PP^{t-1}}(Q)^{b_i},$$
for all~$D\in  \Pic(K)_{-}$ and~$i\in \{2,\ldots,n\}$. Together with Corollary~\ref{cor:varphiOfOplus}, these estimates imply
$$\varphi\left(D\otimes \overline{E_Q} \right) \leq 2^nC_1(e|\Delta_K|^{\frac{1}{2}})^{n-1}H_{\PP^{t-1}}(Q)^{|\mathbf{b}|-b_1},$$
for all~$D\in \mathrm{I}_Q$. Hence, by Lemma~\ref{lem: convergencia integral varphi(D)^m}(2) we conclude
\begin{equation}\label{eq first integral over IQ}
  \int_{ \mathrm{I}_Q}\N(D)^{-s}\varphi_m(D)\varphi\left(D\otimes \overline{E_Q} \right) \,\textup{d}D=O(H_{\PP^{t-1}}(Q)^{|\mathbf{b}|-b_1}).  
\end{equation}

Now, by Proposition~\ref{prop:cota_Geer_Schoof} there exists a constant~$\beta\geq 1$, depending only on~$K$, such that~$\varphi(D)\leq \beta e^{-\pi n_Ke^{-\frac{2}{n_K}\widehat{\deg}(D)}}$ for all~$D\in \Pic(K)_{-}$, with~$n_K=[K:\QQ]$ as usual. Together with Lemma~\ref{lem: convergencia integral varphi(D)^m}(1), this implies 
\begin{equation*}
    \begin{split}
       \left| \int_{ \mathrm{I}_Q}\N(D)^{n-s}\varphi_m(D) \,\textup{d}D\right| & \leq \beta' \int_{-\infty}^{\frac{1}{2}\log(|\Delta_K|)-b_1\log(H_{\PP^{t-1}}(Q))} e^{x(n- \Re(s))-\pi n_K e^{-\frac{2}{n_K}x}}\textup{d}x,
    \end{split}
\end{equation*}
for some constant~$\beta'>0$ depending only on~$K$ and~$m$. Put~$C_2:=\frac{|\mathbf{b}|}{b_1}$, and let~$T<0$ such that
$$x(n- \Re(s))-\pi n_K e^{-\frac{2}{n_K}x}\leq C_2x \quad \text{for all }x\in ]-\infty,T] \text{ and }s\in \mathcal{K}.$$
If~$\frac{1}{2}\log(|\Delta_K|)-b_1\log(H_{\PP^{t-1}}(Q))\leq T$, then
\begin{equation*}
    \begin{split}
       \left| \int_{ \mathrm{I}_Q}\N(D)^{n-s}\varphi_m(D) \,\textup{d}D\right| & \leq \beta' \int_{-\infty}^{\frac{1}{2}\log(|\Delta_K|)-b_1\log(H_{\PP^{t-1}}(Q))} e^{C_2x}\textup{d}x\\
       &=\frac{\beta' (|\Delta_K|)^{\frac{C_2}{2}}}{C_2 H_{\PP^{t-1}}(Q)^{C_2b_1}}.
    \end{split}
\end{equation*}
Hence, on the one hand, assuming~$\frac{1}{2}\log(|\Delta_K|)-b_1\log(H_{\PP^{t-1}}(Q))\leq T$ we get
\begin{equation}\label{eq second integral over IQ}
   H_{\PP^{t-1}}(Q)^{|\mathbf{b}|} \left| \int_{\mathrm{I}_Q}\N(D)^{n-s}\varphi_m(D) \,\textup{d}D\right|
=O(1).  
\end{equation}
On the other hand, if~$\frac{1}{2}\log(|\Delta_K|)-b_1\log(H_{\PP^{t-1}}(Q))\geq T$ then~$H_{\PP^{t-1}}(Q)$ is bounded above by a constant that depends only on~$K$ and~$T$, hence in that case~\eqref{eq second integral over IQ} 
also holds thanks to Lemma~\ref{lem: convergencia integral varphi(D)^m}(2). 

Now we look at integrals over~$ \mathrm{II}_Q$. First, we use Theorem~\ref{thm:RiemannRoch} to get
\begin{equation*}
    \begin{split}
 \varphi\left(D\otimes \overline{E_Q} \right)&=(\varphi\left(D^{\vee}\otimes \overline{E_Q}^{\vee}\otimes \overline{\omega_{\CO_K}} \right)+1)\N(D)^nH_{\PP^{t-1}}(Q)^{|\mathbf{b}|}|\Delta_K|^{-\frac{n}{2}}-1  \\
 &=\N(D)^nH_{\PP^{t-1}}(Q)^{|\mathbf{b}|}|\Delta_K|^{-\frac{n}{2}}-1+\N(D)^nH_{\PP^{t-1}}(Q)^{|\mathbf{b}|}|\Delta_K|^{-\frac{n}{2}}\varphi\left(D^{\vee}\otimes \overline{E_Q}^{\vee}\otimes \overline{\omega_{\CO_K}} \right).
    \end{split}
\end{equation*}
This implies
\begin{equation}\label{eq Riemann Roch over II}
\begin{split}
&\int_{\mathrm{II}_Q}\N(D)^{-s}\varphi_m(D)\varphi\left(D\otimes \overline{E_Q}\right) \,\textup{d}D\\
&\hspace{0.5cm} = |\Delta_K|^{-\frac{n}{2}} H_{\PP^{t-1}}(Q)^{|\mathbf{b}|}\int_{ \mathrm{II}_Q}\N(D)^{n-s}\varphi_m(D)\varphi\left(D^{\vee}\otimes \overline{E_Q}^{\vee} \otimes \overline{\omega_{\CO_K}} \right) \,\textup{d}D\\
&\hspace{1cm} + |\Delta_K|^{-\frac{n}{2}}H_{\PP^{t-1}}(Q)^{|\mathbf{b}|}\int_{\mathrm{II}_Q}\N(D)^{n-s}\varphi_m(D) \,\textup{d}D  \\
&\hspace{1.5cm} -\int_{ \mathrm{II}_Q}\N(D)^{-s}\varphi_m(D)\,\textup{d}D.
\end{split}
\end{equation}
From Lemma~\ref{lem: convergencia integral varphi(D)^m}(2) we know that
\begin{equation}\label{eq third integral over II}
    \int_{ \mathrm{II}_Q}\N(D)^{-s}\varphi_m(D)\,\textup{d}D =O(1).
\end{equation}
Now, note that for~$D\in  \mathrm{II}_Q$ we have~$D^{\vee}\otimes (\overline{E_Q^{(i)}})^\vee \otimes \overline{\omega_{\CO_K}}\in \Pic(K)_{-}$ for all~$i\in\{1,\ldots,n\}$, hence by Proposition~\ref{prop:cota_Geer_Schoof} we get
$$\varphi(D^{\vee}\otimes (\overline{E_Q^{(i)}})^\vee \otimes \overline{\omega_{\CO_K}})\leq \beta e^{-\pi n_K(|\Delta_K|^{-1}\N(D)H_{\PP^{t-1}}(Q)^{b_i})^{\frac{2}{n_K}}}\leq \beta e^{-C_3(\N(D)H_{\PP^{t-1}}(Q)^{b_1})^{\frac{2}{n_K}}},$$
where~$\beta\geq 1$ is the same constant as before and~$C_3:=\pi n_K |\Delta_K|^{-\frac{2}{n_K}}$. Combined with Corollary~\ref{cor:varphiOfOplus}, we conclude
$$\varphi(D^{\vee}\otimes \overline{E_Q}^\vee \otimes \overline{\omega_{\CO_K}})\leq 2^n\beta^n e^{-nC_3(\N(D)H_{\PP^{t-1}}(Q)^{b_1})^{\frac{2}{n_K}}},$$
for all~$D\in  \mathrm{II}_Q$. Letting~$C_4>0$ be such that~$xe^{-nC_3x^{\frac{2}{n_K}}}\leq C_4$ for all $x\geq 0$, we get
$$\N(D)H_{\PP^{t-1}}(Q)^{b_1}e^{-nC_3(\N(D)H_{\PP^{t-1}}(Q)^{b_1})^{\frac{2}{n_K}}}\leq C_4,$$
thus
\begin{equation}\label{eq first integral over II}
  \begin{split}
      &H_{\PP^{t-1}}(Q)^{|\mathbf{b}|}\int_{ \mathrm{II}_Q}\N(D)^{n-s}\varphi_m(D)\varphi\left(D^{\vee}\otimes E^{\vee}_Q\otimes \overline{\omega_{\CO_K}} \right) \,\textup{d}D\\
       & \hspace{1cm} \leq C_4 2^n\beta^n H_{\PP^{t-1}}(Q)^{|\mathbf{b}|-b_1}\int_{ \mathrm{II}_Q}\N(D)^{n-1-s}\varphi_m(D) \,\textup{d}D =O(H_{\PP^{t-1}}(Q)^{|\mathbf{b}|-b_1}),
  \end{split}  
\end{equation}
by Lemma~\ref{lem: convergencia integral varphi(D)^m}(2). The desired result then follows from~\eqref{eq first integral over IQ}, \eqref{eq second integral over IQ}, \eqref{eq Riemann Roch over II}, \eqref{eq third integral over II} and~\eqref{eq first integral over II}. This completes the proof of the lemma.
\end{proof}

\begin{lem}\label{lem: bound for negative split bundles}
Given negative integers~$0>b_1\geq b_2\geq \cdots \geq b_n$ and~$Q\in \PP^{t-1}$, define
$$\overline{E_Q}:=\overline{ \CO_{\PP^{t-1}}(b_1)_Q}\oplus \cdots\oplus  \overline{(\CO_{\PP^{t-1}}(b_n))_Q}.$$
Then, there exist~$\beta,\gamma>0$, depending only on~$K$ and~$n$, such that for every~$D\in \Pic(K)_{-}$ we have
$$\varphi(D\otimes \overline{E_Q})\leq \beta e^{-\gamma e^{-\frac{2}{n_K}\widehat{\deg}(D)}}e^{-\gamma H_{\PP^{t-1}}(Q)^{\frac{2}{n_K}}},$$
where~$n_K=[K:\QQ]$ as usual.
\end{lem}
\begin{proof}
By Corollary~\ref{cor:varphiOfOplus} it is enough to prove the lemma in the case~$n=1$.
Since
$$\widehat{\deg}(D\otimes \overline{\CO_{\PP^{t-1}}(b_1)_Q})=\widehat{\deg}(D)+b_1\log (H_{\PP^{t-1}}(Q))\leq  \frac{1}{2}\log(|\Delta_K|),$$
we can use Proposition~\ref{prop:boundForNumberSections}. This shows that there exists a constant~$\beta>0$, depending only on~$K$, such that
$$\varphi(D\otimes \overline{\CO_{\PP^{t-1}}(b_1)_Q})\leq \beta e^{-\pi n_Ke^{-\frac{2}{n_K}(\widehat{\deg}(D)+b_1\log (H_{\PP^{t-1}}(Q)))}}=\beta e^{-\pi n_Ke^{-\frac{2}{n_K}\widehat{\deg}(D)}H_{\PP^{t-1}}(Q)^{\frac{2}{n_K}}},$$
where in the last inequality we used that~$b_1\leq -1$. Putting~$A:=e^{-\frac{2}{n_K}\widehat{\deg}(D)}$ and $B:=H_{\PP^{t-1}}(Q)^{\frac{2}{n_K}}$, we see that~$A\geq |\Delta_K|^{-\frac{1}{n_K}}$ and~$B\geq 1$, hence there exists a constant~$\rho>0$, depending only on~$K$, such that~$AB\geq \rho(A+B)$. This implies
$$\varphi(D\otimes \overline{\CO_{\PP^{t-1}}(b_1)_Q})\leq \beta e^{-\gamma e^{-\frac{2}{n_K}\widehat{\deg}(D)}}e^{-\gamma H_{\PP^{t-1}}(Q)^{\frac{2}{n_K}}},$$
with~$\gamma=\pi n_K \rho$. This proves the lemma.
\end{proof}

\begin{notation}
    In the proof of Theorem~\ref{thm:generalVersionTheorem} below, we put~$\varphi(D^{\oplus m}):=0$ when~$m=0$.
\end{notation}

\begin{proof}[Proof of Theorem~\ref{thm:generalVersionTheorem}]
Given $P\in X(K)$ we put $Q:=\pi(P)$. Then, we have~$P\in \PP(\mathscr{W}_Q)(K)$  with~$\mathscr{W}$ defined in~\eqref{eq def of W}. Moreover, if $P\in F(K)$ then~$P\in \PP(\mathscr{Y}^\vee_Q)(K)$ and $H_{\PP(\mathscr{W}_Q)}(P)=H_{\PP(\mathscr{Y}^\vee_Q)}(P)$.  Indeed, this follows from the fact that $\mathscr{Y}^\vee_Q\subseteq \mathscr{W}_Q$, which implies that $\overline{\CO_{\mathscr{Y}^\vee_Q}(-1)_P} =\overline{\CO_{\mathscr{W}_Q}(-1)_P}$. Then, taking duals and norms leads to the equality of heights. From this and~\eqref{eq H_L of P}, we get
	\begin{equation*}
		\begin{split}
			\Z_{U,L}(s)
			&=\sum_{P\in U(K)}\left( H_{\PP(\mathscr{W}_{\pi(P)})}(P)^{\lambda }H_{\PP ^{t-1}}\left( \pi(P)\right)^{\mu}\right) ^{-s}\\
			&=\sum_{Q\in \PP^{t-1}(K)}H_{\PP ^{t-1}}\left( Q\right)^{- \mu s} \left(\sum_{P\in \PP(\mathscr{W}_Q)(K)}H_{\PP(\mathscr{W}_Q)}(P)^{-\lambda s} -
   \sum_{P\in \PP(\mathscr{Y}^\vee_Q)(K)}H_{\PP(\mathscr{W}_Q)}(P)^{-\lambda s}\right)\\
			&=\sum_{Q\in \PP^{t-1}(K)}H_{\PP ^{t-1}}\left( Q\right)^{-\mu s} \left( \Z_{\PP(\mathscr{W}_Q)}(\lambda s)-\Z_{\PP(\mathscr{Y}^\vee_Q)}(\lambda s)\right).
		\end{split}
	\end{equation*}
	Now, fixing $Q\in \PP^{t-1}(K)$ and using 
 Proposition \ref{prop:LemmaOfIntegraloverPic}(2), we have 
	\begin{equation*}
		\begin{split}
	       &w_K\xi_K(\lambda s)\Z_{\PP(\mathscr{W}_Q)}(\lambda s)\\
			&= R_Kh_K|\Delta_K|^{-\frac{\lambda s}{2}} \left( \frac{\N(\overline{\mathscr{W}_Q})}{\lambda s-(r+1)}-\frac{1}{\lambda s}\right)+\int_{\Pic(K)_{-}} \N(D)^{-\lambda s}\varphi(D\otimes \overline{\mathscr{W}_Q})\,\textup{d}D\\
            &\hspace{2cm}+ \N(\overline{\mathscr{W}_Q})|\Delta_K|^{\frac{r+1}{2}-\lambda s}\int_{\Pic(K)_{-}}\N(D)^{\lambda s-(r+1)}\varphi(D\otimes \overline{\mathscr{W}_Q}^\vee)\,\textup{d}D.
		\end{split}
	\end{equation*}
Similarly, since $\mathscr{Y}^\vee$ has rank $r$, we have
	\begin{equation*}
		\begin{split}
			&w_K\xi_K(\lambda s)\Z_{\PP(\mathscr{Y}^\vee_Q)}(\lambda s)\\
			&=R_Kh_K|\Delta_K|^{-\frac{\lambda s}{2}} \left(\frac{\N(\overline{\mathscr{Y}^\vee_Q})}{\lambda s-r}-\frac{1}{\lambda s}\right)+\int_{\Pic(K)_{-}} \N(D)^{-\lambda s}\varphi(D\otimes \overline{\mathscr{Y}^\vee_Q})\,\textup{d}D\\
            &\hspace{2cm}+ \N(\overline{\mathscr{Y}^\vee_Q})|\Delta_K|^{\frac{r}{2}-\lambda s}\int_{\Pic(K)_{-}}\N(D)^{\lambda s-r}\varphi(D\otimes \overline{\mathscr{Y}_Q})\,\textup{d}D.
		\end{split}
	\end{equation*}
 	Hence, we can write
    $$w_K\xi_K(\lambda s)\Z_{U,L}(s)=\sum_{j=1}^5F_{j}(s),$$
    where
	\begin{equation*}
		\begin{split}
			&F_1(s):=\frac{R_Kh_K|\Delta_K|^{-\frac{\lambda s}{2}}}{\lambda  s-(r+1)}\sum_{Q\in \PP^{t-1}(K)}H_{\PP ^{t-1}}\left( Q\right)^{-\mu s}\N(\overline{\mathscr{W}_Q}) ,\\
			&F_2(s):=-\frac{R_Kh_K|\Delta_K|^{-\frac{\lambda s}{2}}}{\lambda s-r}\sum_{Q\in \PP^{t-1}(K)}H_{\PP ^{t-1}}\left( Q\right)^{-\mu s}\N(\overline{\mathscr{Y}^\vee_Q}),\\
			&F_3(s):=\sum_{Q\in \PP^{t-1}(K)}H_{\PP ^{t-1}}\left( Q\right)^{-\mu s} \int_{\Pic(K)_{-}} \N(D)^{-\lambda s}\left( \varphi(D\otimes \overline{\mathscr{W}_Q})-\varphi(D\otimes \overline{\mathscr{Y}^\vee_Q})\right) \,\textup{d}D,\\
			&F_4(s):=|\Delta_K|^{\frac{r+1}{2}-\lambda s}\sum_{Q\in \PP^{t-1}(K)}H_{\PP ^{t-1}}\left( Q\right)^{-\mu s}\N(\overline{\mathscr{W}_Q})\int_{\Pic(K)_{-}}\N(D)^{\lambda s-(r+1)}\varphi(D\otimes \overline{\mathscr{W}_Q}^\vee)\,\textup{d}D,\\
			&F_5(s):=-|\Delta_K|^{\frac{r}{2}-\lambda s}\sum_{Q\in \PP^{t-1}(K)}H_{\PP ^{t-1}}\left( Q\right)^{-\mu s}\N(\overline{\mathscr{Y}^\vee_Q})\int_{\Pic(K)_{-}}\N(D)^{\lambda s-r}\varphi(D\otimes \overline{\mathscr{Y}_Q})\,\textup{d}D.\\
		\end{split}
	\end{equation*}
We are going to analyze each of these functions separately.  First, we compute
 	\begin{equation}\label{eq N(overline{mathscr{W}_Q})}
		\begin{split}
  \N(\overline{\mathscr{Y}^\vee_Q})=\N(\overline{\mathscr{W}_Q})&=\N(\overline{(\CO_{\PP^{t-1}})_{Q}} \oplus \overline{\CO_{\PP^{t-1}}(a_r)_{Q}}\oplus  \cdots\oplus \overline{\CO_{\PP^{t-1}}(a_r-a_{r-1})_{Q}})\\
			&=H_{\PP^{t-1}}(Q)^{(r+1)a_r-|\mathbf{a}|}.
		\end{split}
	\end{equation}
This implies
\begin{equation*}\label{F_1}
	F_1(s)=\frac{R_Kh_K|\Delta_K|^{-\frac{\lambda s}{2}}}{\lambda s-(r+1)} \Z_{\mathbb{P}^{t-1}}\left( \mu s+|\mathbf{a}|-(r+1)a_r\right),
\end{equation*}
hence by Theorem~\ref{thm Maruyama} the function~$F_1(s)$ is holomorphic in~$\Re(s)>\frac{1+(r+1)a_r-|\mathbf{a}|}{\mu},s\neq \lambda_L,s\neq \mu_L$. Similarly,
 \begin{equation*}\label{F_2}
        F_2(s)=-\frac{R_Kh_K|\Delta_K|^{-\frac{\lambda s}{2}}}{\lambda s-r} \Z_{\mathbb{P}^{t-1}}\left( \mu s+|\mathbf{a}|-(r+1)a_r\right),
 \end{equation*}
hence~$F_2(s)$ is holomorphic in~$\Re(s)>\frac{1+(r+1)a_r-|\mathbf{a}|}{\mu },s\neq \frac{r}{\lambda},s\neq \mu_L$.

We now focus on the function~$F_3(s)$. First, note that $\overline{\mathscr{W}_Q}=\overline{\mathscr{Y}_Q^\vee} \oplus \overline{(\CO_{\PP^{t-1}})_Q}$ and hence, using Lemma \ref{lem:SectionsOfDirectSum}(2), we can compute
	\begin{equation*}
		\begin{split}
			\varphi(D\otimes \overline{\mathscr{W}_Q})
   &= \varphi(D\otimes (\overline{\mathscr{Y}^\vee_Q}\oplus \overline{(\CO_{\PP^{t-1}})_Q}))\\
			&= \varphi(D\otimes \overline{\mathscr{Y}^\vee_Q})+\varphi(D\otimes \overline{(\CO_{\PP^{t-1}})_Q} )+\varphi(D\otimes \overline{\mathscr{Y}^\vee_Q})\varphi(D\otimes \overline{(\CO_{\PP^{t-1}})_Q})\\
			&=\varphi(D\otimes \overline{\mathscr{Y}^\vee_Q})+\varphi(D)+\varphi(D\otimes \overline{\mathscr{Y}^\vee_Q})\varphi(D).
		\end{split}
	\end{equation*}
	We deduce that $F_3(s)=G_1(s)+G_2(s),$ where
	\begin{equation*}
		\begin{split}
			G_1(s):=&\sum_{Q\in \PP^{t-1}(K)}H_{\PP ^{t-1}}\left( Q\right)^{-\mu s} \int_{\Pic(K)_{-}} \N(D)^{-\lambda s}\varphi(D)\,\textup{d}D\\
			=&\Z_{\mathbb{P}^{t-1}}\left( \mu s \right)\int_{\Pic(K)_{-}} \N(D)^{-\lambda s}\varphi(D)\,\textup{d}D,
		\end{split}
	\end{equation*}
	and 
	$$G_2(s):=\sum_{Q\in \PP^{t-1}(K)}H_{\PP ^{t-1}}\left( Q\right)^{-\mu s} \int_{\Pic(K)_{-}} \N(D)^{-\lambda s}\varphi(D\otimes \overline{\mathscr{Y}^\vee_Q})\varphi(D)\,\textup{d}D.$$
By Lemma~\ref{lem: convergencia integral varphi(D)^m}(1) with~$m=0$, together with Theorem~\ref{thm Maruyama}, the function~$G_1(s)$ is holomorphic in~$\Re(s)>\frac{t}{\mu}$. In order to analyze the function~$G_2(s)$, put $a_0:=0$ an recall that~$N_X=\#\{i\in \{1,\ldots,r\}:a_i=a_r\}$, so we can write
$$\overline{\mathscr{Y}_Q^\vee} =\overline{(\CO_{\PP^{t-1}})_Q}^{\oplus (N_X-1)}\oplus \overline{E_Q}$$
where~$\overline{E_Q}$ is the direct sum of the line bundles~$\overline{\CO_{\PP^{t-1}}(a_r-a_i))_Q}$ over~$i\in\{0,\ldots,r-1\}$ with~$a_i<a_r$. Using~Lemma \ref{lem:SectionsOfDirectSum}(2) again, we have
\begin{equation*}
    \begin{split}
     \varphi(D\otimes \overline{\mathscr{Y}_Q^\vee})
     &=\varphi(D^{\oplus (N_X-1)})+\varphi(D\otimes \overline{E_Q})+\varphi(D^{\oplus (N_X-1)})\varphi(D\otimes \overline{E_Q}).  
    \end{split}
\end{equation*}
We now use Lemma~\ref{lem: bound for positive split bundles} with~$n=r-(N_X-1)$, $b_1=a_r-a_{r-N_X},b_2=a_r-a_{r-N_X-1},\ldots,b_n=a_r$, and~$m=0$, and also with~$m=N_X-1$ if~$N_X>1$, in order to write
\begin{equation*}
    \begin{split}
   G_2(s)
   &= \widetilde{G_2}(s) + |\Delta_K|^{\frac{N_X-(r+1)}{2}}\sum_{Q\in \PP^{t-1}(K)}H_{\PP ^{t-1}}\left( Q\right)^{-\mu s+(r+1)a_r-|\mathbf{a}|}\\
   & \hspace{5cm} \times \int_{\Pic(K)_{-}}\N(D)^{r-(N_X-1)-\lambda s}\varphi(D)(1+\varphi(D^{\oplus (N_X-1)})) \,\textup{d}D\\
   &= \widetilde{G_2}(s)+|\Delta_K|^{\frac{N_X-(r+1)}{2}}\Z_{\PP^{t-1}}(\mu s+|\mathbf{a}|-(r+1)a_r)\\
   & \hspace{5cm} \times \int_{\Pic(K)_{-}}\N(D)^{r-(N_X-1)-\lambda s}\varphi(D)(1+\varphi(D^{\oplus (N_X-1)})) \,\textup{d}D,
    \end{split}
\end{equation*}
with~$\widetilde{G_2}(s)$ an analytic function on~$\Re(s)>\frac{t+(r+1)a_r-|\mathbf{a}|-b_1}{\mu}=\frac{t+ra_r+a_{r-N_X}-|\mathbf{a}|}{\mu}$. Hence,
\begin{equation*}
    \begin{split}
   F_3(s)
   &= G_1(s)+ \widetilde{G_2}(s)+|\Delta_K|^{\frac{N_X-(r+1)}{2}}\Z_{\PP^{t-1}}(\mu s+|\mathbf{a}|-(r+1)a_r)\\
   & \hspace{5cm} \times \int_{\Pic(K)_{-}}\N(D)^{r-(N_X-1)-\lambda s}\varphi(D)(1+\varphi(D^{\oplus (N_X-1)})) \,\textup{d}D,
    \end{split}
\end{equation*}
with~$G_1(s)+ \widetilde{G_2}(s)$ analytic in~$\Re(s)>
\frac{t+ra_r+a_{r-N_X}-|\mathbf{a}|}{\mu}$.

In order to analyze the function~$F_4(s)$, we start by using~\eqref{eq N(overline{mathscr{W}_Q})}  to write
\begin{equation*}
		\begin{split}
			F_4(s)&=|\Delta_K|^{\frac{r+1}{2}-\lambda s }\sum_{Q\in \PP^{t-1}(K)}H_{\PP ^{t-1}}\left( Q\right)^{-\mu s+(r+1)a_r-|\mathbf{a}|}\\
			&\hspace{4cm}\times\int_{\Pic(K)_{-}}\N(D)^{\lambda s-(r+1)}\varphi(D\otimes \overline{\mathscr{W}_Q^\vee})\,\textup{d}D.
		\end{split}
	\end{equation*} 
%
Now we write
$$\overline{\mathscr{W}_Q^\vee} =\overline{(\CO_{\PP^{t-1}})_Q}^{\oplus N_X}\oplus \overline{E_Q^{\vee}}$$
with~$\overline{E_Q^{\vee}}$ the sum of the line bundles~$\overline{\CO_{\PP^{t-1}}(a_i-a_r)_Q}$ over~$i\in\{0,\ldots,r-1\}$ with~$a_i<a_r$. By Lemma~\ref{lem:SectionsOfDirectSum}(2) we have
\begin{equation*}
    \begin{split}
     \varphi(D\otimes \overline{\mathscr{W}_Q^\vee})
     &=\varphi(D^{\oplus N_X})+\varphi(D\otimes \overline{E_Q^\vee})+\varphi(D^{\oplus N_X})\varphi(D\otimes \overline{E_Q^\vee}).  
    \end{split}
\end{equation*}
Hence, we can write~$F_4(s)=G_3(s)+G_4(s)$ where
\begin{equation*}
    \begin{split}
    G_3(s)&:=|\Delta_K|^{\frac{r+1}{2}-\lambda s}\Z_{\mathbb{P}^{t-1}}\left(\mu s+|\mathbf{a}|-(r+1)a_r  \right)  \int_{\Pic(K)_{-}}\N(D)^{\lambda s-(r+1)}\varphi(D^{\oplus N_X})\,\textup{d}D ,\\
    G_4(s) &:=|\Delta_K|^{\frac{r+1}{2}-\lambda s}\sum_{Q\in \PP^{t-1}(K)}H_{\PP ^{t-1}}\left( Q\right)^{-\mu s+(r+1)a_r-|\mathbf{a}|}\\
			&\hspace{3.5cm}\times\int_{\Pic(K)_{-}}\N(D)^{\lambda s-(r+1)}\left( \varphi(D\otimes \overline{E_Q^\vee})+\varphi(D^{\oplus N_X})\varphi(D\otimes \overline{E_Q^\vee}) \right)\,\textup{d}D.
    \end{split}
\end{equation*}
On the one hand, by Theorem~\ref{thm Maruyama} and Proposition~\ref{prop:LemmaOfIntegraloverPic}(1), the series~$G_3(s)$ extends to a holomorphic function in~$\Re(s)>\frac{1+(r+1)a_r-|\mathbf{a}|}{\mu},s\neq \mu_L$. On the other hand, 
using Lemmas~\ref{lem: bound for negative split bundles} and~\ref{lem: convergencia integral varphi(D)^m}(1) to bound~$\varphi(D\otimes \overline{E_Q^\vee})$ and~$\varphi(D^{\oplus N_X})$, respectively, we see that the series~$G_4(s)$ converges absolutely and uniformly for~$s$ in compact subsets of~$\CC$, hence~$G_4(s)$ extends to an entire function. We conclude that~$F_4(s)$ is holomorphic in~$\Re(s)>\frac{1+(r+1)a_r-|\mathbf{a}|}{\mu},s\neq \mu_L$.

Finally, the analysis of the function~$F_5(s)$ is analogous to that of~$F_4(s)$, and we get that~$F_5(s)=G_5(s)+G_6(s)$ with
\begin{equation*}
   G_5(s):= -|\Delta_K|^{\frac{r}{2}-\lambda s}\Z_{\mathbb{P}^{t-1}}\left(\mu s+|\mathbf{a}|- (r+1)a_r  \right) 
     \int_{\Pic(K)_{-}}\N(D)^{\lambda s-r}\varphi(D^{\oplus (N_X-1)})\,\textup{d}D,
     \end{equation*}
and~$G_6(s)$ entire. In particular, $F_5(s)$ is holomorphic in~$\Re(s)>\frac{1+(r+1)a_r-|\mathbf{a}|}{\mu},s\neq \mu_L$.

Putting everything together, and recalling that~$t\geq 2$ and~$a_{r}\geq 1$, we obtain the following:
 \begin{enumerate}
			\item If $\lambda_L=\mu_L$,
   then~$\Z_{U,L}(s)$ is holomorphic in
$$\Re(s)>\max\left\{ \frac{1+(r+1)a_r-|\mathbf{a}|}{\mu},\frac{r}{\lambda},\frac{t+ra_r+a_{r-N_X}-|\mathbf{a}|}{\mu}\right\},s\neq \lambda_L$$
and it has a pole of order two at~$s=\lambda_L$ (coming from~$F_1$) with
$$\lim_{s\to \lambda_L}(s-\lambda_L)^2\Z_{U,L}(s)=\frac{R_Kh_K|\Delta_K|^{-\frac{r+1}{2}}}{w_K\lambda \mu \xi(r+1)} \operatorname{Res}_{s=t}\Z_{\PP^{t-1}}(s)=\frac{R_K^2h_K^2|\Delta_K|^{-\frac{(d+2)}{2}}}{w_K^2 \lambda \mu \xi_K(r+1)\xi_K(t) }.$$
			\item If $\lambda_L>\mu_L$,
			then $\Z_{U,L}(s)$ has holomorphic continuation to
   $$\Re(s)>\max\left\{ \frac{r}{\lambda},\mu_L \right\},s\neq \lambda_L$$
   and it has a simple pole at~$s=\lambda_L$ (coming from~$F_1$) with 
   $$\lim_{s\to \lambda_L}(s-\lambda_L)\Z_{U,L}(s)=\frac{R_Kh_K|\Delta_K|^{-\frac{r+1}{2}}}{w_K \lambda \xi(r+1)} \Z_{\mathbb{P}^{t-1}}\left( \mu \lambda_L+|\mathbf{a}|-(r+1)a_r\right).$$

			\item If~$\lambda_L<\mu_L$ 
			then $\Z_{U,L}(s)$ is holomorphic in
   $$\Re(s)> \max\left\{ \lambda_L,\frac{1+(r+1)a_r-|\mathbf{a}|}{\mu},\frac{t+ra_r+a_{r-N_X}-|\mathbf{a}|}{\mu}\right\} ,s\neq \mu_L$$
   and it has a possible singularity at~$s=\mu_L$ (coming from~$F_1,F_2,F_3,G_3$ and~$G_5$) with
\begin{equation*}
    \begin{split}
\lim_{s\to \mu_L}(s-\mu_L)\Z_{U,L}(s)&=\frac{R_Kh_K|\Delta_K|^{-\frac{t}{2}}}{w_K^2 \mu \xi_K(\lambda \mu_L)\xi_K(t)}\left( \frac{R_Kh_K |\Delta_K|^{-\frac{\lambda \mu_L}{2}}}{\lambda\mu_L -(r+1)} -\frac{R_Kh_K |\Delta_K|^{-\frac{\lambda \mu_L}{2}}}{\lambda\mu_L -r} \right.    \\
& \hspace{0.5cm} +|\Delta_K|^{\frac{N_X-(r+1)}{2}} \int_{\Pic(K)_{-}}\N(D)^{r-(N_X-1)-\lambda \mu_L}\varphi(D)(1+\varphi(D^{\oplus (N_X-1)}) )\,\textup{d}D \\
& \hspace{1cm} +|\Delta_K|^{\frac{r+1}{2}-\lambda \mu_L}\int_{\Pic(K)_{-}}\N(D)^{\lambda \mu_L-(r+1)}\varphi(D^{\oplus N_X})\,\textup{d}D\\
& \hspace{1.5cm}\left. -|\Delta_K|^{\frac{r}{2}-\lambda \mu_L}\int_{\Pic(K)_{-}}\N(D)^{\lambda \mu_L-r}\varphi(D^{\oplus (N_X-1)})\,\textup{d}D \right).
    \end{split}
\end{equation*}
Furthermore, in this case we can write (using~Lemma~\ref{lem:SectionsOfDirectSum}(2))
$$\varphi(D)(1+\varphi(D^{\oplus (N_X-1)}))=\varphi(D^{\oplus N_X})-\varphi(D^{\oplus (N_X-1)}),$$
and use Proposition~\ref{prop:LemmaOfIntegraloverPic}(2), together with formula~\eqref{eq extension of xi_K} when~$N_X=1$ or~$2$, to get
\begin{equation*}
    \begin{split}
\lim_{s\to \mu_L}(s-\mu_L)\Z_{U,L}(s)&=\frac{R_Kh_K|\Delta_K|^{-\frac{t}{2}}}{w_K^2 \mu \xi_K(\lambda \mu_L)\xi_K(t)} w_K |\Delta_K|^{\frac{N_X-(r+1)}{2}}\xi_K(\lambda \mu_L+N_X-(r+1))\\
& \hspace{1cm}\times \big( \Z_{\PP^{N_X-1}}(\lambda \mu_L+N_X-(r+1))-\Z_{\PP^{N_X-2}}(\lambda \mu_L+N_X-(r+1)) \big). 
    \end{split}
\end{equation*}
Since this value is positive, we conclude that~$\Z_{U,L}(s)$ has a simple pole at~$s=\mu_L$ in this case. 
 		\end{enumerate}
Then, the asymptotic formula for~$N(U,H|_L,B)$ follows from these properties, together with~\eqref{eq a(L) and b(L)} and Theorem~\ref{thm:tauberianThm}. This completes the proof of the theorem.
\end{proof}

\begin{remark}
The different cases that appear in Theorem~\ref{thm:generalVersionTheorem} give a subdivision of the \emph{big cone} of $X$, i.e., the interior of~$\Lambda_{\textup{eff}}$ (see Figure~\ref{figure} for an illustration). The line bundles~$L$ contained in the ray passing through the anticanonical class have height zeta functions with a double pole at $s=\lambda_L=\mu_L$, while line bundles outside this ray have~$\lambda_L\neq \mu_L$ and have height zeta functions with a simple pole at $s=\max\{\lambda_L,\mu_L\}$.
\end{remark}

\begin{figure}[h]
    \centering
    \begin{tikzpicture}[x=0.75pt,y=0.75pt,yscale=-1,xscale=1]

\draw  (173.17,247.44) -- (455.5,247.44) (202.56,36.71) -- (202.56,268) (448.5,242.44) -- (455.5,247.44) -- (448.5,252.44) (197.56,43.71) -- (202.56,36.71) -- (207.56,43.71)  ;
\draw    (202.56,247.44) -- (382.01,43.83) ;
\draw [shift={(383.33,42.33)}, rotate = 131.39] [color={rgb, 255:red, 0; green, 0; blue, 0 }  ][line width=0.75]    (10.93,-3.29) .. controls (6.95,-1.4) and (3.31,-0.3) .. (0,0) .. controls (3.31,0.3) and (6.95,1.4) .. (10.93,3.29)   ;
\fill  [color={rgb, 255:red, 208; green, 2; blue, 27 }  ,fill opacity=0.17 ] (383.33,42.33) -- (202.56,247.44) --  (451.32,247.44) -- (451.32,42.33) -- cycle ;
\fill [color={rgb, 255:red, 74; green, 144; blue, 226 }  ,fill opacity=0.28 ] (202.56,247.44) -- (382.9,42.56) -- (201.88,43.1) -- cycle ;
\draw    (296,141) -- (202.56,247.44) ;
\draw [shift={(296,141)}, rotate = 131.28] [color={rgb, 255:red, 0; green, 0; blue, 0 }  ][fill={rgb, 255:red, 0; green, 0; blue, 0 }  ][line width=0.75]      (0, 0) circle [x radius= 2.34, y radius= 2.34]   ;

\draw (177.41,29.97) node [anchor=north west][inner sep=0.75pt]    {$\mu $};
\draw (469.49,244.98) node [anchor=north west][inner sep=0.75pt]    {$\lambda $};
\draw (288.95,157.62) node [anchor=north west][inner sep=0.75pt]  [color={rgb, 255:red, 0; green, 0; blue, 0 }  ,opacity=1 ]  {$-K_{X}$};
\draw (348,16.07) node [anchor=north west][inner sep=0.75pt]  [font=\scriptsize]  {$\mu =\left(\frac{( r+1) a_{r} -|\mathbf{a} |+t}{r+1}\right) \lambda $};
\draw (327.33,116.73) node [anchor=north west][inner sep=0.75pt]  [font=\scriptsize,color={rgb, 255:red, 208; green, 2; blue, 27 }  ,opacity=1 ]  {$\mu <\left(\frac{( r+1) a_{r} -|\mathbf{a} |+t}{r+1}\right) \lambda $};
\draw (212,58.07) node [anchor=north west][inner sep=0.75pt]  [font=\scriptsize,color={rgb, 255:red, 49; green, 0; blue, 255 }  ,opacity=1 ]  {$\mu  >\left(\frac{( r+1) a_{r} -|\mathbf{a} |+t}{r+1}\right) \lambda $};
\end{tikzpicture}
    \caption{Subdivision of the big cone of $X$. }
    \label{figure}
\end{figure}
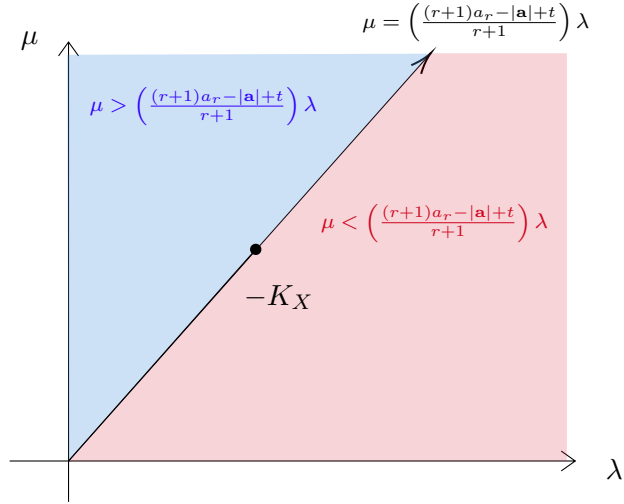

In Theorem~\ref{thm:generalVersionTheorem} we have omitted the case when~$a_r=0$. This is because in the case~$a_r=0$ we have~$X_d(a_1,\ldots,a_r)\simeq \PP^{t-1}\times \PP^r$, and there is no need to remove a proper subvariety of~$X$ to obtain the ``correct'' growth of the number of rational points of bounded height. Note that, in this case, we have
$$\lambda_L=\frac{r+1}{\lambda}, \quad \mu_L=\frac{t}{\mu}.$$

\begin{thm}\label{thm:degenerate_case_general_theorem}
     Let $X\simeq \PP^{t-1}\times \PP^r$ be a Hirzebruch--Kleinschmidt variety over the number field~$K$ with~$a_r=0$, and let~$L=\lambda h+\mu f$ be a big line bundle class in~$\Pic(X)$. Then, we have
        $$N(X,H_L,B)\sim C_{L,K}B^{a(L)}\log(B)^{b(L)} \quad \text{as }B\to \infty,$$
     with~$C_{L,K}$ given by
     $$\left\{
     \begin{array}{ll}
      \frac{R_K^2h_K^2|\Delta_K|^{-\frac{(d+2)}{2}}}{w_K^2 (r+1) \mu \xi_K(r+1)\xi_K(t)}    &  \text{if }\lambda_L=\mu_L, \\
       \frac{R_Kh_K|\Delta_K|^{-\frac{r+1}{2}}}{w_K (r+1) \xi(r+1)} \Z_{\mathbb{P}^{t-1}}\left( \mu \lambda_L\right)   & \text{if }\lambda_L>\mu_L,\\
          \frac{R_Kh_K|\Delta_K|^{-\frac{t}{2}}}{w_K t \xi_K(t)}\Z_{\PP^{r}}(\lambda \mu_L) 
          & \text{if }\lambda_L<\mu_L.
     \end{array}\right.$$
\end{thm}

One can adapt the proof of Theorem~\ref{thm:generalVersionTheorem} to give a proof of Theorem~\ref{thm:degenerate_case_general_theorem}. Instead of doing that, we present a simpler argument based only on the analytic properties of the height zeta functions of the projective spaces~$\PP^{t-1}$ and~$\PP^r$.

\begin{proof}
    We have
    $$\Z_{X,L}(s):=\sum_{P\in X(K)}H_L(P)^{-s}=\Z_{\PP^{r}}(\lambda s)\Z_{\PP^{t-1}}(\mu s).$$
    Then, using Theorem~\ref{thm Maruyama} we see that~$\Z_{X,L}(s)$ is holomorphic in
    $$\Re(s)>\max\left\{ \frac{1}{\lambda},\frac{1}{\mu}\right\},s\neq \lambda_L,s\neq \mu_L,$$
    and it has a double pole at~$s=\lambda_L$ if~$\lambda_L=\mu_L$, and a simple pole at~$s=\max\{\lambda_L,\mu_L\}$ if~$\lambda_L\neq \mu_L$. Then, the result follows by using Theorem~\ref{thm Maruyama} to compute~$\lim_{s\to a(L)}(s-a(L))^{b(L)}Z_{X,L}(s)$, and using Theorem~\ref{thm:tauberianThm}. We leave the details to the reader.
\end{proof}

\begin{example}\label{example product of projective spaces}
    Consider the variety~$X_2(0)\simeq \PP^1\times \PP^1$ and~$L=3h+f$, so that~$\lambda_L=\frac{2}{3}$ and~$\mu_L=2$. Then, we get
   $$N(X,H_L,B)\sim C_{L,K}B^{2} \quad \text{as }B\to \infty,$$
   with~$C_{L,K}= \frac{R_Kh_K|\Delta_K|^{-1}}{w_K 2 \xi_K(2)}\Z_{\PP^{1}}(6)$. In the case~$K=\QQ$, a simple computation gives
   $$\Z_{\PP^1}(s)= 2+2\frac{\zeta(s/2)}{\zeta(s)}L_{-4}(s/2), \text{ where } L_{-4}(s):=\sum_{n=1}^{\infty}\left(\frac{-4}{n}\right)n^{-s}.$$
   Hence, using that~$L_{-4}(3)=\frac{\pi^3}{32}$ (see e.g~\cite[p.~189]{Coh07}) and~$\zeta(6)=\frac{\pi^6}{945}$, we get
   $$C_{L,\QQ}=\frac{6}{\pi}\left(1+\frac{945\zeta(3)}{32\pi^3}\right)=4.09640530\ldots .$$ 
   Note that~$C_{L,\QQ}=C'+C''$ where~$C',C''$ are the constants appearing at the end of Example~\ref{example HK threefold} in the Introduction (because~$X_2(0)\simeq U'\sqcup F'$ in the notation used there).
\end{example}

\subsection{The anticanonical height - Proof of Theorem~\ref{main thm anticanonical}}\label{subsection proof of main thm anticanonical}

In the case~$\lambda=r+1$ and $\mu=(r+1)a_r+t-|\mathbf{a}|$ we get~$L=-K_X$ by Proposition~\ref{prop effective cone}(2), hence~$H_L$ is the anticanonial height function~$H=H_{-K_X}$. Moreover, $\lambda_L=\mu_L=1$ according to~\eqref{eq lambda_D and mu_D}. When~$a_r>0$, Theorem~\ref{thm:generalVersionTheorem} gives the asymptotic formula
$$N(U,H,B)\sim C B\log(B) \quad \text{as }B\to \infty,$$
with~$C$ given by~\eqref{eq constant main thm anticanonical}. Assume~$a_r=0$. Then~$X\simeq \PP^{t-1}\times \PP^r$ and~$F\simeq \PP^{t-1}\times \PP^{r-1}$ if~$r\geq 2$, while~$F\simeq \PP^{t-1}$ if~$r=1$. By Lemmas~\ref{lem:restrictionBigDivisor} and~\ref{lem:restrictionBigDivisorII}, together with Theorem~\ref{thm:degenerate_case_general_theorem} and Corollary~\ref{cor:SchanuelTheorem}, we have
\begin{equation*}
    \begin{split}
    N(X,H,B) &\sim CB\log(B),\\
    N(F,H,B) &\sim C_rB 
    \end{split}
\end{equation*}
as~$B\to \infty$, with the same~$C$ as before and~$C_r$ another explicit constant\footnote{The exact value is~$C_r=\frac{R_Kh_K|\Delta_K|^{-\frac{t}{2}}}{w_K t \xi_K(t)}\Z_{\PP^{r-1}}\left(r+1\right)$ if~$r\geq 2$, and~$\frac{R_Kh_K|\Delta_K|^{-\frac{t}{2}}}{w_K t \xi_K(t)}$ if~$r=1$.}. In any case, this implies
$$N(U,H,B)=N(X,H,B)-N(F,H,B)\sim CB\log(B) \quad \text{as }B\to \infty.$$
This proves Theorem~\ref{main thm anticanonical}.

\subsection{Accumulation of rational points}\label{sec accumulation of rational points}

In the literature, there are different notions that capture the idea of subvarieties having too many rational points. Since our aim in this paper is to give explicit asymptotic formulas that allow for quantitative comparisons, we introduce the following relative notion of subvarieties accumulating more rational points than others.

\begin{defi}
    Let~$Y_1,Y_2\subseteq X$ be two subvarieties of a Hirzebruch--Kleinschmidt variety~$X$, and let~$L\in \Pic(X)$ be a big line bundle class. Assume that~$\#Y_1(K)=\#Y_2(K)=\infty$ and that~$N(Y_1,H_L,B)$ and~$N(Y_2,H_L,B)$ are both finite for every~$B>0$. Then, we say that~$Y_1$ \emph{strongly accumulates more rational points of bounded~$H_L$-height than~$Y_2$} if
    \begin{equation}\label{eq strongly accumulates more}
    \lim_{B\to \infty}\frac{N(Y_2,H_L,B)}{N(Y_1,H_L,B)}=0.\end{equation}
\end{defi}

Theorems~\ref{thm:generalVersionTheorem} and~\ref{thm:degenerate_case_general_theorem} lead to the following corollary.

\begin{cor}\label{cor accumulating subvarieties}
    Let $X=X_d(a_1,\ldots,a_r)$ be a Hirzebruch--Kleinschmidt variety over the number field~$K$ with~$a_r>0$ and good open subset~$U=U_d(a_1,\ldots,a_r)$, and let~$L=\lambda h+\mu f$ be a big line bundle class in~$\Pic(X)$. Then, the following properties hold:
    \begin{enumerate}
        \item If~$r>1$, assume~$\mu>\lambda (a_r-a_{r-1})$ so that~$L|_F$ is big on the subvariety~$F\simeq X_{d-1}(a_1,\ldots,a_{r-1})$ of~$X$. Then, the the good open subset~$U'\simeq U_{d-1}(a_1,\ldots,a_{r-1})$ of~$F$ strongly accumulates more rational points of bounded~$H_L$-height than~$U$ if and only if
        \begin{equation}\label{eq accumulating criterion}
        \max\{\lambda_L,\mu_L\}<\mu_{L|_F},
        \end{equation}
        where~$\mu_{L|_F}=\frac{ra_{r-1}+t-|\mathbf{a}|+a_r}{\mu-\lambda  (a_r-a_{r-1})}$.
        \item If~$r=1$, assume~$\mu>\lambda a_1$ so that~$L|_F$ is big on~$F\simeq \PP^{t-1}$. Then, $F$ strongly accumulates more rational points of bounded~$H_L$-height than~$U$ if and only if~\eqref{eq accumulating criterion} holds, where~$\mu_{L|_F}=\frac{t}{\mu-a_1}$.
    \end{enumerate}
\end{cor}
\begin{proof}
    Assume~$r>1$ and~$\mu>\lambda (a_r-a_{r-1})$. By Theorem~\ref{thm:generalVersionTheorem} the good open subset~$U'$ of~$F$ strongly accumulates more rational points of bounded~$H_L$-height than~$U$ if and only if
    $$\max\{\lambda_L,\mu_L\}<\max\{\lambda_{L|_F},\mu_{L|_F}\}, \text{ or }\max\{\lambda_L,\mu_L\}=\lambda_{L|_F}=\mu_{L|_F} \text{ and }\lambda_L\neq \mu_L.$$
    Since~$\lambda_{L|_F}=\frac{r}{\lambda}<\frac{r+1}{\lambda}=\lambda_L$, we see that the second case cannot occur. This proves~$(1)$.
    
    Now, if~$r=1$ and~$\mu>\lambda a_1$, then by Lemma~\ref{lem:restrictionBigDivisorII} the height function~$H_L$ restricted to~$F$ corresponds to the power~$H_{\PP^{t-1}}^{\mu-\lambda a_1}$ of the standard height function of~$\PP^{t-1}$. By Theorems~\ref{thm:generalVersionTheorem} and Corollary~\ref{cor:SchanuelTheorem} we see that~$F$ strongly accumulates more rational points of bounded~$H_L$-height than~$U$ if and only if~\eqref{eq accumulating criterion} holds. This completes the proof of the corollary.
\end{proof}

\begin{remark}
One can also define a weaker relative notion of subvarieties accumulating more rational points than others, by replacing condition~\eqref{eq strongly accumulates more} with
$$0<\limsup_{B\to \infty}\frac{N(Y_2,H_L,B)}{N(Y_1,H_L,B)}<1.$$
Then, in all possible cases, one can use the explicit formulas in Theorems~\ref{thm:generalVersionTheorem}, \ref{thm:degenerate_case_general_theorem} and Corollary~\ref{cor:SchanuelTheorem} to decide when~$U'$ (resp.~$F$) weakly accumulates more rational points of bounded~$H_L$-height than~$U$. For instance, in the case~$(a_1,a_2)=(0,1)$ of Example~\ref{example HK threefold}, we saw that~$U'$ weakly accumulates more rational points of bounded anticanonical height than~$F'$. 
\end{remark}

\section{Example: Hirzebruch surfaces}\label{subsection hirzebruch surfaces}

For an integer~$a>0$ consider the \emph{Hirzebruch surface}
$$X=X_2(a)=\PP(\CO_{\PP^1}\oplus \CO_{\PP^1}(-a)),$$
which we consider as a variety over~$K=\QQ$ for simplicity. In the basis~$\{h,f\}$ of~$\Pic(X)$ given in Proposition~\ref{prop effective cone}, choose a big line bundle class~$L=\lambda h+\mu f$. We then have
$$\lambda_L=\frac{2}{\lambda}, \quad \mu_L=\frac{a+2}{\mu}.$$
We write~$U=U_2(a)$ and~$F=\PP(\CO_{\PP^1}(-a))$, so that
$$X=U\sqcup F\simeq U_2(a)\sqcup \PP^1.$$
Using that~$\xi_{\mathbb{Q}}(s)=(2\pi^{s/2})^{-1}\Gamma(s/2)\zeta(s)$ and~$\zeta(2)=\frac{\pi^2}{6}$, we get by Theorem \ref{thm:generalVersionTheorem} the asymptotic formula
      $$
	N(U,H_L, B)\sim C_L \left\lbrace \begin{array}{ll}
	  B^{\lambda_L}\log(B),	& \textup{if }\lambda_L=\mu_L,  \\
  B^{\max\{\lambda_L,\mu_L\}},	& \textup{if } \lambda_L\neq \mu_L,  \\
	\end{array}\right.$$
 as~$B\to \infty$, where
 \begin{equation}\label{eq C_L Hirzebruch surface}
 C_L=\left\{
     \begin{array}{ll}
      \frac{18}{\pi^2 \mu}    &  \text{if }\lambda_L=\mu_L, \\
       \frac{3}{\pi} \Z_{\mathbb{P}^{1}}\left( \mu \lambda_L-a\right)   & \text{if }\lambda_L>\mu_L,\\
          \frac{6\xi_{\mathbb{Q}}(\lambda \mu_L-1)}{\pi (a+2) \xi_{\mathbb{Q}}(\lambda \mu_L)}
          & \text{if }\lambda_L<\mu_L.
     \end{array}\right.
     \end{equation}
     Now, the restriction~$L|_F$ of the line bundle class~$L$ to~$F \simeq \PP^1$ is big if and only if~$\mu>\lambda a$. In this case, by Corollary~\ref{cor:SchanuelTheorem}, we have
     $$N(F,H_L,B)\sim \frac{3}{\pi}B^{\frac{2}{\mu-\lambda a}} \quad \text{as }B\to \infty.$$
 As in Corollary~\ref{cor accumulating subvarieties}(2), we have that~$F$ strongly accumulates more points of bounded $H_L$-height than~$U$ if and only if 
 $$\max\left\{ \frac{2}{\lambda},\frac{a+2}{\mu}\right\}<\frac{2}{\mu-\lambda a},$$
 which is easily seen to be equivalent
 $$\lambda a<\mu< \lambda (a+1).$$

Finally, we turn our attention to the numerical value of the constant~$C_L$ in~\eqref{eq C_L Hirzebruch surface}. Assuming~$a=1$ for simplicity, we get the following first values of~$C_L$ depending on the choice of~$L=\lambda h+\mu f$. Here, as in Example~\ref{example product of projective spaces}, we use the Dirichlet $L$-function~$L_{-4}(s):=\sum_{n=1}^{\infty}\left(\frac{-4}{n}\right)n^{-s}$.

\begin{center}
\renewcommand{\arraystretch}{1.5}
         \begin{tabular}{|c|c|c|c|}
         \hline 
         $\lambda$ & $\mu $ &  {\bf Case}  & $C_L$\\
         \hline 
         $1$ & $1$ & $\lambda_L < \mu_L$  &  $\frac{2\xi_{\QQ}(2)}{\pi \xi_{\QQ}(3)}=\frac{2 \pi}{3 \zeta(3)}=1.74234272\ldots $   \\ 
          \hline
         $1$ & $2$ & $\lambda_L > \mu_L$  & $\frac{3}{\pi}\Z_{\PP^1}(3)=\frac{6}{\pi}\left(1+\frac{ \zeta(3/2) L_{-4}(3/2) }{\zeta(3)}\right)=5.49807267\ldots$   \\ 
         \hline
         $1$ & $3$ & $\lambda_L > \mu_L$ &   $\frac{3}{\pi}\Z_{\PP^1}(5)=\frac{6}{\pi}\left(1+\frac{ \zeta(5/2) L_{-4}(5/2) }{\zeta(5)}\right)=4.25372490\ldots$   \\ 
         \hline 
           $2$ & $1$ &  $\lambda_L < \mu_L$  &   $\frac{2\xi_{\QQ}(5)}{\pi \xi_{\QQ}(6)}=\frac{2835 \zeta(5)}{4 \pi^6}=0.76443811\ldots$   \\ 
          \hline
         $2$ & $2$ & $\lambda_L < \mu_L$ &   $\frac{2\xi_{\QQ}(2)}{\pi \xi_{\QQ}(3)}=\frac{2 \pi}{3 \zeta(3)}=1.74234272\ldots\ldots $  \\ 
         \hline
         $2$ & $3$ & $\lambda_L = \mu_L$ &  $\frac{6}{\pi^2}=0.60792710\ldots $   \\
         \hline
         $3$ & $1$ & $\lambda_L < \mu_L$  &    $\frac{2\xi_{\QQ}(8)}{\pi \xi_{\QQ}(9)}=\frac{32 \pi^7}{165375 \zeta(9)}=0.58325419\ldots $  \\ 
          \hline
         $3$ & $2$ & $\lambda_L < \mu_L$  & $\frac{2\xi_{\QQ}(7/2)}{\pi \xi_{\QQ}(9/2)}=\frac{2 \zeta(7/2) \Gamma(7/4)}{\sqrt{\pi} \zeta(9/2) \Gamma(9/4)} =0.97781868\ldots$     \\ 
         \hline
         $3$ & $3$ & $\lambda_L < \mu_L$ &   $\frac{2\xi_{\QQ}(2)}{\pi \xi_{\QQ}(3)}=\frac{2 \pi}{3 \zeta(3)}=1.74234272\ldots\ldots $  \\ 
         \hline
                 \end{tabular}
                 \renewcommand{\arraystretch}{1}
                 \end{center}

\begin{remark}\label{remark Serre and BM}
The case~$X_2(1)= \PP(\CO_{\PP^1}\oplus \CO_{\PP^1}(-1))$ was already studied by Serre (see~\cite[Section~2.12]{Serre1989}), and revisited by Batyrev and Manin in~\cite[Section~1.6]{Bat/Man90} and by Peyre in~\cite[\emph{Proposition}~2.7]{Pey02}. Following the notation in~\cite{Pey02}, one realizes~$X_2(1)$ as the variety
    $$V=\left\{([y_0,y_1,y_2],[z_0,z_1])\in \PP^2\times \PP^1:y_0z_1=y_1z_0\right\}.$$
    Then, for integers~$r,s$, the height function~$H_{r,s}$ on~$V(\QQ)$ defined by
    $$H_{r,s}(([y_0,y_1,y_2],[z_0,z_1])):=\sqrt{y_0^2+y_1^2+y_2^2}^{r+s}\sqrt{z_0^2+z_1^2}^{-s},$$
    for~$(y_0,y_1,y_2)\in \ZZ^3$ and~$(z_0,z_1)\in \ZZ^2$ primitive, corresponds to our height function~$H_L$ with~$L=(r+s)h+rf$. Moreover, the divisor~$E\subset V$ defined in loc.~cit.~by~$y_0=y_1=0$ corresponds to our subbundle~$F$. In the particular case of~$r=1$ and~$s=0$, corresponding to~$\lambda=\mu=1$, we see that
    \begin{equation*}
       \begin{split}
           N(V\setminus E,H_{1,0},B) &\sim \frac{1}{2}\#\left\{(y_0,y_1,y_2)\in \ZZ^3 \text{ primitive, such that }\sqrt{y_0^2+y_1^2+y_2^2}\leq B\right\}\\
           &\sim \frac{2\pi}{3\zeta(3)}B^3 \quad \text{as }B\to \infty
       \end{split} 
    \end{equation*}
(see, e.g.~\cite{CC07}). This matches our computations, since in the case~$\lambda=\mu=1$ we get
$$N(U,H_L,B)\sim C_LB^3, \quad C_L=\frac{2\xi_{\QQ}(2)}{\pi \xi_{\QQ}(3)}=\frac{2\pi}{3\zeta(3)}.$$
\end{remark}

	\bibliographystyle{alpha}
	\bibliography{biblio}
	
\end{document}